\documentclass[12pt, leqno]{article}

\usepackage{amsmath}
\usepackage{amsthm}
\usepackage{amssymb}
\usepackage{mathrsfs}
\usepackage{stmaryrd}
\usepackage{enumitem}
\newlist{condenum}{enumerate}{1}
\usepackage{hyperref}
\usepackage{graphicx, xcolor}
\usepackage{fancyhdr}
\usepackage{tikz}
\usetikzlibrary{patterns, arrows}

\usepackage{makecell}

\usepackage[
	isbn=false, 
	url = false,
	natbib=true
]{biblatex}
\definecolor{specialgreen}{RGB}{1, 140, 38}
\definecolor{specialgold}{RGB}{237, 174, 0}
\definecolor{specialpink}{RGB}{255, 137, 139}
\definecolor{specialblue}{RGB}{138, 183, 255}

\newcounter{intro}

\newtheorem{introthm}[intro]{Theorem}

\newtheorem{thm}{Theorem}[section]
\newtheorem{prop}[thm]{Proposition}
\newtheorem{lem}[thm]{Lemma}
\newtheorem{cor}[thm]{Corollary}
\theoremstyle{definition}
\newtheorem{ex}[thm]{Example}
\theoremstyle{remark}
\newtheorem*{rem}{Remark}
\newtheorem*{rems}{Remarks}
\numberwithin{equation}{section}

\newcommand{\FF}{\mathbb{F}}

\newcommand{\ZZ}{\mathbb{Z}}

\def\mydefb#1{\expandafter\def\csname #1#1#1\endcsname{\mathcal{#1}}}
\def\mydefallb#1{\ifx#1\mydefallb\else\mydefb#1\expandafter\mydefallb\fi}
\mydefallb ABCDEFGHIJKLMNOPQRSTUVWXYZ\mydefallb

\newcommand{\UT}{\mathrm{UT}}

\newcommand{\ut}{\mathfrak{ut}}

\renewcommand{\epsilon}{\varepsilon}

\renewcommand{\setminus}{-}

\newcommand{\qstirling}{\genfrac{\{}{\}}{0pt}{}}
\newcommand{\qbinom}{\genfrac{(}{)}{0pt}{}}

\newcommand{\Ht}{\operatorname{ht}}

\newcommand{\up}[1]{\mathchoice{\hspace{.25em}\uparrow\hspace{-.25em}(#1)}{\hspace{.25em}\uparrow\hspace{-.125em}(#1)}{\uparrow(#1)}{\uparrow(#1)}}
\newcommand{\upnum}[2]{\mathchoice{\hspace{.25em}\uparrow_{#2}\hspace{-.25em}(#1)}{\hspace{.25em}\uparrow_{#2}\hspace{-.25em}(#1)}{\uparrow_{#2}(#1)}{\uparrow_{#2}(#1)}}

\newcommand{\edge}[3][color = black]{\mathchoice{\tikz[scale = 0.75, baseline = -0.1cm]{
\draw[thick, #1] (0, 0) -- (1, 0);
\draw[fill = black] (0, 0) circle (2pt) node[below] {$\scriptstyle #2$}; 
\draw[fill = black] (1, 0) circle (2pt) node[below] {$\scriptstyle #3$}; }}{(#2, #3)}{(#2, #3)}
{(#2, #3)}}

\newcommand{\dedge}[3][color = specialgold, dashed]{\mathchoice{\tikz[scale = 0.75, baseline = -0.15cm]{
\draw[thick, #1] (0, 0.5) -- (0, -0.5);
\draw[fill = black] (0, 0.5) circle (2pt) node[left] {$\scriptstyle #2$}; 
\draw[fill = black] (0, -0.5) circle (2pt) node[left] {$\scriptstyle #3$};}}{(#2, #3)}{(#2, #3)}{(#2, #3)}}

\newcommand{\binding}[5][color = specialgold, dashed]{\tikz[scale = 0.75, baseline = -0.15cm]{
\draw[fill] (-0.5, 0.5) circle (2pt) node[inner sep = 1.5pt] (i) {};
\draw[fill] (0.5, 0.5) circle (2pt)  node[inner sep = 1.5pt] (j) {};
\draw[fill] (-0.5, -0.5) circle (2pt) node[inner sep = 1.5pt] (k) {};
\draw[fill] (0.5, -0.5) circle (2pt) node[inner sep = 1.5pt] (l) {};
\node[above] at (i) {$\scriptstyle #2$};
\node[above] at (j) {$\scriptstyle #3$}; 
\node[below] at (k) {$\scriptstyle #4$};
\node[below] at (l) {$\scriptstyle #5$};
\draw[thick, #1] (i) -- (k);  
\draw[thick, #1] (j) -- (l); 
\draw[thick] (i) -- (j);
\draw[thick] (k) -- (l);}}

\newcommand{\subbinding}[5][color = specialgold, dashed]{\tikz[scale = 0.25, baseline = -0.05cm]{
\draw[fill] (-0.5, 0.5) circle (2pt) node[inner sep = 1.5pt] (i) {};
\draw[fill] (0.5, 0.5) circle (2pt)  node[inner sep = 1.5pt] (j) {};
\draw[fill] (-0.5, -0.5) circle (2pt) node[inner sep = 1.5pt] (k) {};
\draw[fill] (0.5, -0.5) circle (2pt) node[inner sep = 1.5pt] (l) {};
\node[left] at (i) {$\scriptscriptstyle #2$};
\node[right] at (j) {$\scriptscriptstyle #3$}; 
\node[left] at (k) {$\scriptscriptstyle #4$};
\node[right] at (l) {$\scriptscriptstyle #5$};
\draw[thick, #1] (i) -- (k);  
\draw[thick, #1] (j) -- (l); 
\draw[thick] (i) -- (j);
\draw[thick] (k) -- (l);}}

\begin{document}

\title{The combinatorics of normal subgroups in the unipotent upper triangular group}
\author{Lucas Gagnon}
\date{}
\maketitle
\begin{abstract}
Describing the conjugacy classes of the unipotent upper triangular groups $\mathrm{UT}_{n}(\mathbb{F}_{q})$ uniformly (for all or many values of $n$ and $q$) is a nearly impossible task.  This paper takes on the related problem of describing the normal subgroups of $\mathrm{UT}_{n}(\mathbb{F}_{q})$.  For $q$ a prime, a bijection will be established between these subgroups and pairs of combinatorial objects with labels from $\mathbb{F}_{q}^{\times}$. Each pair comprises a loopless binary matroid and a tight splice, an apparently new kind of combinatorial object which interpolates between nonnesting set partitions and shortened polyominoes.  For arbitrary $q$, the same approach describes a natural subset of normal subgroups: those which correspond to the ideals of the Lie algebra $\mathfrak{ut}_{n}(\mathbb{F}_{q})$ under an approximation of the exponential map.
\end{abstract}

\section{Introduction}
\label{sec:intro}

The unipotent upper triangular group $\UT_{n}(\FF_{q})$ consists of all upper triangular matrices with each diagonal entry equal to $1$ over the finite field $\FF_{q}$ with $q$ elements.  Uniformly indexing the conjugacy classes or irreducible characters of $\UT_{n}(\FF_{q})$ (for every $n$ and $q$) is an impossibly difficult problem~\cite{GudivokEtAl}.  However, developments in unipotent combinatorics~\cite{AguiarEtAl, AliniaeifardThiem20, Andrews, Thiem} suggest that the difficulty of the conjugacy problem belies nicer combinatorial structures present in the more computable, conjugacy-adjacent properties of $\UT_{n}(\FF_{q})$.

Absent a good understanding of the fundamental representation theoretic structures of $\UT_{n}(\FF_{q})$, two questions arise: what methods can access the important properties of $\UT_{n}(\FF_{q})$, and is there a nice, combinatorial description for the output of these methods?  Answers to the former are often algebraic, like the Kirillov orbit method~\cite{Kirillov, Sangroniz}, or recursive, like the character-theoretic techniques in~\cite{Isaacs95, Isaacs07, Marberg12a}.  Answers to the latter---when positive---are bijections with families of combinatorial objects; examples include supercharacter theories with classes and characters indexed by labeled set partitions~\cite{AguiarEtAl} or nonnesting set partitions~\cite{AliniaeifardThiem20, Andrews}, and an indexing of certain irreducible characters with labeled lattice paths~\cite{Marberg12b}.

The aim of this paper is to determine whether the normal subgroups of $\UT_{n}(\FF_{q})$ belong to the latter universe of bijections and combinatorial objects.  Normal subgroups are closely related to conjugacy classes, so it is not clear that the normal subgroups of $\UT_{n}(\FF_{q})$ should have a uniform indexing set, and even if such a set exists, a nice exposition remains nontrivial.  

What is known about the normal subgroups of $\UT_{n}(\FF_{q})$ follows the classical intuition.  For a Lie group $G$, the exponential map $\exp$ gives a bijection between closed normal subgroups of $G$ and ideals of the Lie algebra $\operatorname{Lie}(G)$. Over $\FF_{q}$, the map $\exp$ is not well-defined and the right notion of closure is unclear, but morally there should be similar correspondences for nice enough groups, as is the case in the Lazard correspondence~\cite{Lazard}, and for the finite groups of Lie type.  In~\cite{Levcuk74} (see also~\cite{Levcuk76}), Lev\v{c}uk gives a bijection of this sort for $\UT_{n}(\FF_{q})$, using the map $a \mapsto 1+a$ from the nilpotent Lie algebra $\ut_{n}(\FF_{q})$ to $\UT_{n}(\FF_{q})$ in place of $\exp$.

\begin{introthm}[{\cite[Theorem 1]{Levcuk74}}]\label{introthm:A}\label{thm:correspondence}
A subset $N \subseteq \UT_{n}(\FF_{q})$ is a normal subgroup if and only if $N = 1 + \mathfrak{n}$ for an additive subgroup $\mathfrak{n} \le \ut_{n}(\FF_{q})$ with $\{[a, b]\;|\; a \in \ut_{n}(\FF_{q}), b \in \mathfrak{n}\} \subseteq \mathfrak{n}$, where $[a, b] = ab - ba$ is the Lie bracket of $\ut_{n}(\FF_{q})$.
\end{introthm}

Lie algebra ideals are $\FF_{q}$-subspaces, so if $q$ is not prime $\UT_{n}(\FF_{q})$ has normal subgroups which do not correspond to Lie algebra ideals. (In fact, these are Lie \textit{ring} ideals.)  Let
\[
\mathscr{I}_{n}(q) = \{1 + \mathfrak{n} \;|\; \text{$\mathfrak{n}$ is an ideal of $\ut_{n}(\FF_{q})$}\} \subseteq \{N \trianglelefteq \UT_{n}(\FF_{q})\}.
\]
The main result of this paper concerns this set of ``closed'' normal subgroups.

\begin{introthm}\label{introthm:B}
There is a bijection
\[
\mathscr{I}_{n}(q) \longleftrightarrow \left\{ (\SSS, \sigma, \MMM, \tau) \;\middle|\; \begin{array}{c}
\text{$(\SSS, \sigma)$ is an $\FF_{q}^{\times}$-labeled tight splice on $\{1, 2, \ldots, n\}$}\\
\text{and $(\MMM, \tau)$ is an $\FF_{q}^{\times}$-labeled loopless binary} \\ 
\text{matroid on the rows and columns of $\SSS$}
\end{array}
 \right\}.
\]
When $q$ is prime, this accounts for every normal subgroup of $\UT_{n}(\FF_{q})$.
\end{introthm}

An explanation of the objects in this bijection follows.  A \emph{tight splice} is an apparently unknown object formed by connecting the blocks of a nonnesting set partition according to certain rules.  The result is a graph in which each connected component has the shape of a \emph{shortened polyomino}, a combinatorial object used in~\cite{SzuEnEtAl07, SzuEnEtAl13, DeutschShapiro} to study Catalan statistics; these statistics appear in my formula for $|\mathscr{I}_{n}(q)|$, see Corollary~\ref{cor:numideals}.  This grid-like shape allows for a labelling scheme in which labels are placed in each ``box'' of a splice.  For example
\begin{center}
\begin{tikzpicture}[baseline = -0.1cm]
\begin{scope}[xshift = 0cm, yshift = 1cm]
\draw[fill] (0, 0) circle (2pt) node[inner sep = 2pt] (1) {};
\draw[fill] (1, 0) circle (2pt) node[inner sep = 2pt] (2) {};
\draw[fill] (2, 0) circle (2pt) node[inner sep = 2pt] (4) {};
\end{scope}
\begin{scope}[xshift = 0cm, yshift = 0cm]
\draw[fill] (0, 0) circle (2pt) node[inner sep = 2pt] (3) {};
\draw[fill] (1, 0) circle (2pt) node[inner sep = 2pt] (5) {};
\draw[fill] (2, 0) circle (2pt) node[inner sep = 2pt] (6) {};
\end{scope}
\begin{scope}[xshift = 1cm, yshift = -1cm]
\draw[fill] (0, 0) circle (2pt) node[inner sep = 2pt] (7) {};
\draw[fill] (1, 0) circle (2pt) node[inner sep = 2pt] (8) {};
\end{scope}
\node[above] at (1) {$\scriptstyle 1$};
\node[above] at (2) {$\scriptstyle 2$};
\node[below] at (3) {$\scriptstyle 3$};
\node[above] at (4) {$\scriptstyle 4$};
\node[below left] at (5) {$\scriptstyle 5$};
\node[below right] at (6) {$\scriptstyle 6$};
\node[below] at (7) {$\scriptstyle 7$};
\node[below] at (8) {$\scriptstyle 8$};
\draw[thick] (1) -- (2);
\draw[thick] (2) -- (4);
\draw[thick] (3) -- (5);
\draw[thick] (5) -- (6);
\draw[thick] (7) -- (8);
\draw[thick, dashed, color=specialgold] (1) -- (3);
\draw[thick, dashed, color=specialgold] (2) -- (5);
\draw[thick, dashed, color=specialgold] (4) -- (6);
\draw[thick, dashed, color=specialgold] (5) -- (7);
\draw[thick, dashed, color=specialgold] (6) -- (8);
\end{tikzpicture}
\qquad $\rightarrow$  \qquad
\begin{tikzpicture}[scale = 1, baseline = -0.1cm]
\begin{scope}[xshift = 0cm, yshift = 1cm]
\draw[fill] (0, 0) circle (2pt) node[inner sep = 2pt] (1) {};
\draw[fill] (1, 0) circle (2pt) node[inner sep = 2pt] (2) {};
\draw[fill] (2, 0) circle (2pt) node[inner sep = 2pt] (4) {};
\path (0.5, 0) -- node[pos = 0.5, color = red] {$4$} (0.5, -1);
\path (1.5, 0) -- node[pos = 0.5, color = red] {$1$} (1.5, -1);
\end{scope}
\begin{scope}[xshift = 0cm, yshift = 0cm]
\draw[fill] (0, 0) circle (2pt) node[inner sep = 2pt] (3) {};
\draw[fill] (1, 0) circle (2pt) node[inner sep = 2pt] (5) {};
\draw[fill] (2, 0) circle (2pt) node[inner sep = 2pt] (6) {};
\path (1.5, 0) -- node[pos = 0.5, color = red] {$2$} (1.5, -1);
\end{scope}
\begin{scope}[xshift = 1cm, yshift = -1cm]
\draw[fill] (0, 0) circle (2pt) node[inner sep = 2pt] (7) {};
\draw[fill] (1, 0) circle (2pt) node[inner sep = 2pt] (8) {};
\end{scope}
\node[above] at (1) {$\scriptstyle 1$};
\node[above] at (2) {$\scriptstyle 2$};
\node[below] at (3) {$\scriptstyle 3$};
\node[above] at (4) {$\scriptstyle 4$};
\node[below left] at (5) {$\scriptstyle 5$};
\node[below right] at (6) {$\scriptstyle 6$};
\node[below] at (7) {$\scriptstyle 7$};
\node[below] at (8) {$\scriptstyle 8$};
\draw[thick] (1) -- (2);
\draw[thick] (2) -- (4);
\draw[thick] (3) -- (5);
\draw[thick] (5) -- (6);
\draw[thick] (7) -- (8);
\draw[thick, dashed, color=specialgold] (1) -- (3);
\draw[thick, dashed, color=specialgold] (2) -- (5);
\draw[thick, dashed, color=specialgold] (4) -- (6);
\draw[thick, dashed, color=specialgold] (5) -- (7);
\draw[thick, dashed, color=specialgold] (6) -- (8);
\end{tikzpicture}
\end{center}
illustrates the labeling of a tight splice with elements of $\FF_{5}^{\times}$ (shown in red).  This tight splice originates from the nonnesting set partition with blocks $\{124|356|78\}$ and has the shape of the shortened polyomino $(EESS, SESE)$. 

The other object in the bijection is a loopless binary matroid with a labeling.  The matroid is given as a bipartite graph (sometimes called a Stanley graph, see Section~\ref{sec:matroids} and~\cite[A135922]{OEIS}) on the set of rows and columns of the tight splice (above, there are two rows and two columns).  The labeling assigns an element of $\FF_{q}^{\times}$ to each edge in this graph.

\begin{rem}
There is a unique $\FF_{2}^{\times}$-labeling for each tight splice and each loopless binary matroid, so in effect there is a bijection with unlabeled objects:
\[
\mathscr{I}_{n}(2) \longleftrightarrow \left\{ (\SSS, \MMM) \;\middle|\; \begin{array}{c}
\text{$\SSS$ is a tight splice on $\{1, 2, \ldots, n\}$ and $\MMM$ is a loop-}\\
\text{less binary matroid on the rows and columns of $\SSS$.} 
\end{array}
\right\}.
\]
\end{rem}

The proof of Theorem~\ref{introthm:B} constructs each ideal of $\ut_{n}(\FF_{q})$ from a unique tuple $(\SSS, \sigma, \MMM, \tau)$.  Each labeled tight splice $(\SSS, \sigma)$ determines a family of ideals via certain shared properties.   Selecting an ideal from this family is a matter of $\FF_{q}$-linear algebra, which can be encoded into a labeled loopless binary matroid $(\MMM, \tau)$.  Considering the set of all normal subgroups in $\UT_{n}(\FF_{q})$ when $q$ is not prime, there is no equivalent encoding.  Accordingly, I am able to give an outline of sorts for the normal subgroups of $\UT_{n}(\FF_{q})$ in terms of tight splices (Corollary~\ref{cor:grpsplicecondition}), but not a nice bijective description.

A point of comparison is Marberg's work in~\cite{Marberg11b} on ``supernormal'' subgroups of $\UT_{n}(\FF_{q})$, which relate to two-sided ideals in the \textit{associative} algebra $\ut_{n}(\FF_{q})$ in the same way that ordinary normal subgroups relate to Lie algebra ideals.  For prime $q$, there is a bijection between supernormal subgroups and pairs comprising a nonnesting set partition and a certain type of $\FF_{q}$-vector space.   Each nonnesting set partition has a trivial tight splice, and in Theorem~\ref{introthm:B}, Marberg's subgroups are given by a trivial tight splice (which has a unique labeling) and an arbitrary labeled loopless binary matroid.  Thus tight splices are a heuristic for the difference in ideal structure between $\ut_{n}(\FF_{q})$ as a Lie algebra and $\ut_{n}(\FF_{q})$ as an associative algebra.  

The paper is organized as follows.  Section~\ref{sec:prelims} is a review of preliminary material, including nonnesting set partitions and matroids.  Section~\ref{sec:splices} is an introduction to tight splices.  In Section~\ref{sec:ideals}, I construct the ideals of $\ut_{n}(\FF_{q})$ from tight splices and loopless binary matroids, proving Theorem~\ref{introthm:B}.  Finally, Section~\ref{sec:normalsubgroups} concerns the normal subgroups of $\UT_{n}(\FF_{q})$, giving a formula for the total number of normal subgroups in $\UT_{n}(\FF_{q})$ (Theorem~\ref{thm:numsubgroups}) and several other results; a proof of Theorem~\ref{introthm:A} is also included for completeness.

My results suggest a few new lines of inquiry.  First, there is the question of normal subgroups in maximal unipotent subgroups for other Lie types; I suspect that my approach will adapt well to the classical cases.  Second, normal subgroups form a sublattice in the lattice of subgroups of $\UT_{n}(\FF_{q})$.  The structure of this sublattice is mostly opaque, but Section~\ref{sec:normalsubgrouplattices} describes the join-irreducible elements of the sublattice.  Finally, the connection between tight splices and shortened polyominoes remains somewhat mysterious.

\paragraph{Acknowledgments}

This paper is part of a Ph.D.~thesis undertaken at the University of Colorado Boulder, and I am extremely grateful for the help and insights of my thesis advisor Nathaniel Thiem throughout.  Additional thanks are due to Farid Aliniaeifard for his role in catalyzing this work and to an anonymous referee whose thoughtful suggestions---especially on the enumerative formulas herin---have greatly improved my exposition.

\section{Preliminaries}
\label{sec:prelims}

This section covers background material on nonnesting set partitions, ideals of $\ut_{n}(\FF_{q})$, and matroids.

\subsection{Nonnesting set partitions}
\label{sec:nnps}

Write $[n]  = \{1, 2, \ldots, n\}$, and let
\[
[[n]] = \Big\{\edge{i}{j}{} = (i, j) \;\Big|\; 1 \le i < j \le n \Big\} \subseteq [n] \times [n],
\]
the set of increasing edges in the complete directed graph on $[n]$.  When drawing elements of $[[n]]$ as edges, I will always orient them from left and right or from top to bottom.

There are a few partial orders on $[[n]]$.  First, the orders $\preceq_{L}$ and $\preceq_{R}$ given by
\[
\edge{i}{j}  \preceq_{L} \edge{r}{s} \quad \text{if $r \le i$ and $j = s$} \qquad\text{and}\qquad \edge{i}{j} \preceq_{R} \edge{r}{s} \quad \text{if $i = r$ and $j \le s$.} 
\]
Let $\preceq$ be the smallest partial order which extends both $\prec_{L}$ and $\prec_{R}$, so that
\[
\edge{i}{j} \preceq \edge{r}{s} \quad \text{if $r \le i < j \le s$.}
\]
This order will be the primary order on $[[n]]$ used in this paper.  The poset on $[[n]]$ under $\preceq$ is graded by the \emph{height function}, $\Ht(\displaystyle \edge{i}{j}) = j - i$.

A \emph{set partition of $[n]$} is a subset $\lambda \subseteq [[n]]$ which is an antichain in both $\preceq_{L}$ and $\preceq_{R}$.  This definition differs from the usual notion of a partition of the set $[n]$---a collection of disjoint subsets whose union is $[n]$---but is equivalent: from an antichain $\lambda \subseteq [[n]]$, the graph on $[n]$ with edge set $\lambda$ has connected components which divide $[n]$ into pairwise disjoint sets.  These connected components also determine $\lambda$ uniquely: since $\lambda$ must be an antichian, there edges will only occur between sequential elements of the same connected component.

The cardinality of a set partition $\lambda$ of $[n]$ is the number of edges in $\lambda$.

\begin{ex}\label{ex:nonnestingpartitions}
Let $\lambda = \{\edge{1}{2}, \edge{3}{5}, \edge{4}{6}\}\}$ and $\mu= \{\edge{1}{2}, \edge{2}{6}, \edge{3}{4}, \edge{4}{5}\}$ be set partitions of $[6]$, so that $|\lambda| = 3$ and $|\mu| = 4$.  As graphs, 
\[
\lambda = 
\begin{tikzpicture}[baseline = -0.1cm]
\draw[fill] (0, 0) circle (2pt) node[inner sep = 1pt] (1) {};
\draw[fill] (1, 0) circle (2pt) node[inner sep = 1pt] (2) {};
\draw[fill] (2, 0) circle (2pt) node[inner sep = 1pt] (3) {};
\draw[fill] (3, 0) circle (2pt) node[inner sep = 1pt] (4) {};
\draw[fill] (4, 0) circle (2pt) node[inner sep = 1pt] (5) {};
\draw[fill] (5, 0) circle (2pt) node[inner sep = 1pt] (6) {};
\node[below] at (1) {$\scriptstyle 1$};
\node[below] at (2) {$\scriptstyle 2$};
\node[below] at (3) {$\scriptstyle 3$};
\node[below] at (4) {$\scriptstyle 4$};
\node[below] at (5) {$\scriptstyle 5$};
\node[below] at (6) {$\scriptstyle 6$};
\draw[thick] (1) to[out = 45, in = 135] (2);
\draw[thick] (3) to[out = 45, in = 135] (5);
\draw[thick] (4) to[out = 45, in = 135] (6);
\end{tikzpicture}
\qquad\text{and}\qquad
\mu = 
\begin{tikzpicture}[baseline = -0.1cm]
\draw[fill] (0, 0) circle (2pt) node[inner sep = 1pt] (1) {};
\draw[fill] (1, 0) circle (2pt) node[inner sep = 1pt] (2) {};
\draw[fill] (2, 0) circle (2pt) node[inner sep = 1pt] (3) {};
\draw[fill] (3, 0) circle (2pt) node[inner sep = 1pt] (4) {};
\draw[fill] (4, 0) circle (2pt) node[inner sep = 1pt] (5) {};
\draw[fill] (5, 0) circle (2pt) node[inner sep = 1pt] (6) {};
\node[below] at (1) {$\scriptstyle 1$};
\node[below] at (2) {$\scriptstyle 2$};
\node[below] at (3) {$\scriptstyle 3$};
\node[below] at (4) {$\scriptstyle 4$};
\node[below] at (5) {$\scriptstyle 5$};
\node[below] at (6) {$\scriptstyle 6$};
\draw[thick] (1) to[out = 45, in = 135] (2);
\draw[thick] (2) to[out = 45, in = 135] (6);
\draw[thick] (3) to[out = 45, in = 135] (4);
\draw[thick] (4) to[out = 45, in = 135] (5);
\end{tikzpicture},
\]
with connected components $\{1, 2\}, \{3, 5\}, \{4, 6\}$ for $\lambda$ and $\{1, 2, 6\}, \{3, 4, 5\}$ for $\mu$.
\end{ex}

If a set partition $\lambda$ of $[n]$ is also an antichain in $\preceq$, say that $\lambda$ is \emph{nonnesting}.  Write
\[
\mathrm{NNSP}_{n} = \{\text{nonnesting set partitions of $[n]$}\}.
\]
For $\lambda \in \mathrm{NNSP}_{n}$, drawing the elements of $[n]$ in increasing order from left to right will ensure that no edge of $\lambda$ starts and ends between the endpoints of another (``nests'').  In Example~\ref{ex:nonnestingpartitions}, $\lambda$ is nonnesting but $\mu$ is not, as both $\edge{3}{4} \prec \edge{2}{6}$ and $\edge{4}{5} \prec \edge{2}{6}$.

A subset $\FFF \subseteq [[n]]$ is a \emph{upper set} (with respect to $\preceq$) if is is upwardly closed: if $\edge{i}{j} \in \FFF$ and $\edge{r}{s} \succeq \edge{i}{j}$, then $\FFF$ must also contain $\edge{r}{s}$.  For $\lambda \subseteq [[n]]$ and $\ell \in \ZZ_{\ge 0}$, let
\[
\upnum{\lambda}{\ell} = \bigcup_{\edge{i}{j} \in \lambda}  \{ {\textstyle \edge{r}{s} \in [[n]] \;|\; \edge{r}{s} \succeq \edge{i}{j},\, \Ht(\edge{r}{s}) \ge \Ht(\edge{i}{j}) + \ell}\},
\]
be the upper set of elements at least $\ell$ covering relations in $\preceq$ above some $\edge{i}{j} \in \lambda$.  An example of this construction can be seen in Figure~\ref{fig:filtercofilter}.

Write $\up{\lambda} = \upnum{\lambda}{0}$ for the upper set generated by the set $\lambda$. The fundamental theorem of finite distributive lattices states that there is a bijection
\[
\begin{array}{rcl}
\mathrm{NNSP}_{n} & \longleftrightarrow & \{\text{upper sets of $[[n]]$}\} \\
\lambda & \longmapsto & \up{\lambda} \\
\min(\FFF) & \longmapsfrom & \FFF
\end{array}
\]
where $\min(\FFF)$ denotes the set of $\preceq$-minimal elements in $\FFF$.

\begin{figure}
\centering
\begin{center}
\includegraphics[page=1]{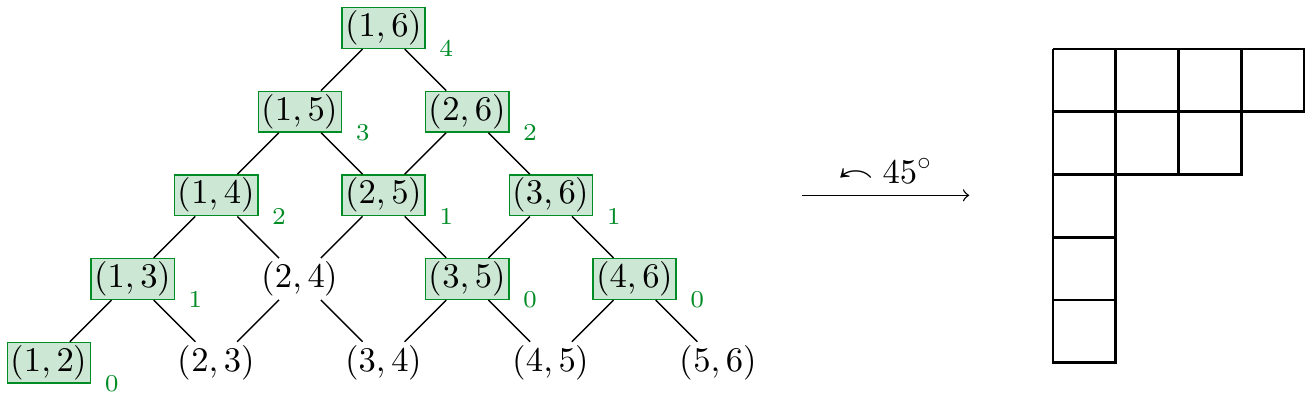}
\end{center}
\caption{The Hasse diagram of $[[6]]$ with elements of the upper set $\up{\lambda}$ generated by the nonnesting set partition $\lambda = \{\edge{1}{2}, \edge{3}{5}, \edge{4}{6}\}$ boxed, and the corresponding Ferrer's shape.  Each $\edge{i}{j} \in \up{\lambda}$ is labeled by the integer $\max\{\ell\in \ZZ_{\ge 0} \;|\; \edge{i}{j} \in \upnum{\lambda}{\ell}\}$.
}
\label{fig:filtercofilter}
\end{figure}

\begin{rem}
The number of upper sets of $[[n]]$ is the $n$th Catalan number $\frac{1}{n+1}\binom{2n}{n}$.  One proof (of many, see~\cite[Problem 178]{StanleyCat}) is as follows:  draw upper sets as in Figure~\ref{fig:filtercofilter} and then rotate counterclockwise by forty-five degrees; the result is a Ferrer's shape which fits inside the shape $(n-1, n-2, \ldots, 1)$.  For $\lambda \in \mathrm{NNSP}_{n}$, each $\edge{i}{j} \in \lambda$ corresponds to a removable corner of the shape associated to $\up{\lambda}$.
\end{rem}

\subsection{Normal subgroups and ideals from nonnesting set partitions}
\label{sec:npsgs}

The set $[[n]]$ indexes the above-diagonal entries of an $n \times n$ matrix, so
\[
\ut_{n}(\FF_{q}) = \bigoplus_{\edge{i}{j} \in [[n]]} \FF_{q} e_{i, j} \qquad \text{with $e_{i, j} = (\delta_{i, r}\delta_{j, s})_{1 \le r, s \le n}$},
\]
where $\delta$ is the Kronecker delta.  The \emph{support} of a subset $\mathfrak{n} \subseteq \ut_{n}(\FF_{q})$ or $1 +  \mathfrak{n} \subseteq \UT_{n}(\FF_{q})$ is
\[
\operatorname{supp}(\mathfrak{n}) = \operatorname{supp}(1 + \mathfrak{n}) = \Big\{\edge{i}{j} \in [[n]] \;\Big|\; \text{$a_{i, j} \neq 0$ for some $a \in \mathfrak{n}$}\Big\}.
\]
It is straightforward to see that this set will be an upper set in $[[n]]$ if $n$ is closed under all inner derivations $b \mapsto [xe_{i, j}, b]$ with $x \in \FF_{q}$ and $\edge{j}{k} \in [[n]]$.  Consequently, $\operatorname{supp}(\mathfrak{n})$ is an upper set if $1 + \mathfrak{n}$ is closed under conjugation by $\UT_{n}(\FF_{q})$, as
\[
(1 + xe_{j, k}) (1+b) (1 + xe_{j, k})^{-1} =  1 + b + [xe_{j, k}, b],
\]
for each $b \in \ut_{n}(\FF_{q})$, $(j, k) \in [[n]]$, and $x \in \FF_{q}$.

For ideals $\mathfrak{n}, \mathfrak{m} \subseteq \ut_{n}$, $\operatorname{supp}(\mathfrak{n} + \mathfrak{m}) = \operatorname{supp}(\mathfrak{n})\cup \operatorname{supp}(\mathfrak{m})$, so for each upper set $\FFF$ there is a maximal ideal
\[
\ut_{\FFF} = \bigoplus_{\edge{i}{j} \in \FFF} \FF_{q}e_{i, j}
\]
having support $\FFF$.  Under $\uparrow$ the set $\mathrm{NNSP}_{n}$ indexes these ideals: let
\[
\ut_{\lambda} = \ut_{\up{\lambda}} \qquad\text{for each $\lambda \in \mathrm{NNSP}_{n}$.}
\]

By~\cite[Lemma 4.1]{Marberg11b}, the set of nonnesting set partitions of $[n]$ (or upper sets of $[[n]]$) also indexes a family of normal subgroups of $\UT_{n}(\FF_{q})$:
\[
\UT_{\lambda} = \UT_{\up{\lambda}} = 1 + \ut_{\up{\lambda}}\qquad\text{for each $\lambda \in \mathrm{NNSP}_{n}$.}
\]
Call these subgroups \emph{normal pattern subgroups}.  If $q \neq 2$ these are the only normal subgroups of $\UT_{n}(\FF_{q})$ which are closed under conjugation by every invertible diagonal matrix~\cite[Proposition 2.1]{DiaconisThiem}.


\begin{ex}\label{ex:patternthings}
Several examples of normal pattern subgroups can be found below.
\begin{enumerate}
\item For $n = 6$ and $\lambda = \{\edge{1}{2}{}, \edge{3}{5}{}, \edge{4}{6}{}\}$ as in Example~\ref{ex:nonnestingpartitions}, the normal pattern subgroup corresponding to $\lambda$ is
\begin{center}
$\UT_{\lambda} = $\;
\begin{tikzpicture}[scale = 0.5, baseline = -1.6cm]
\draw[thick] (0.3, 0.1) -- (-0, 0.1) -- (-0, -6.1) -- (0.3, -6.1);
\draw[thick] (5.8, 0.1) -- (6.1, 0.1) -- (6.1, -6.1) -- (5.8, -6.1);
\draw[thick] (0.2, 0) -- (1, 0) -- (1, -1) -- (4, -1) -- (4, -3) -- (5, -3) -- (5, -4) -- (6, -4) -- (6, -6);
\node at (1.5, -0.5) {$\ast$};
\node at (2.5, -0.5) {$\ast$};
\node at (3.5, -0.5) {$\ast$};
\node at (4.5, -0.5) {$\ast$};
\node at (5.5, -0.5) {$\ast$};
\node at (2.5, -1.5) {$0$};
\node at (3.5, -1.5) {$0$};
\node at (4.5, -1.5) {$\ast$};
\node at (5.5, -1.5) {$\ast$};
\node at (3.5, -2.5) {$0$};
\node at (4.5, -2.5) {$\ast$};
\node at (5.5, -2.5) {$\ast$};
\node at (4.5, -3.5) {$0$};
\node at (5.5, -3.5) {$\ast$};
\node at (5.5, -4.5) {$0$};
\foreach \x in {1, 2, 3, 4, 5}{\pgfmathsetmacro{\z}{\x-1} \foreach \y in {0, ..., \z}{ \node at ({\y + 0.5}, {-\x - 0.5}) {$0$};}}
\foreach \x in {0, 1, 2, 3, 4, 5}{\node at ({\x+0.5}, {-\x-0.5}) {$1$};}
\end{tikzpicture}\;.
\end{center}
I have drawn a line below the entries corresponding to elements of $\up{\lambda}$ to emphasize the connection between this upper set and $\UT_{\lambda}$.

\item Consider the nonnesting set partitions $\emptyset$, $\lambda = \{\edge{1}{n}\}$, and $\mu = \{\edge{i}{i+1} \;|\; 1 \le i < n\}$.  The corresponding normal pattern subgroups are
\[
\UT_{\emptyset} = 1, \quad \UT_{\lambda} = Z(\UT_{n}(\FF_{q})),\quad\text{and}\quad \UT_{\mu} = \UT_{n}(\FF_{q}),
\]
where $Z(\UT_{n}(\FF_{q}))$ is the center of $\UT_{n}(\FF_{q})$.

\end{enumerate}
\end{ex}

The sets $\upnum{\lambda}{\ell}$ have an interpretation in terms of ideals and normal pattern subgroups:
\[
 \ut_{\upnum{\lambda}{\ell}}  = \underbrace{[\ut_{n}, [\ut_{n}, \cdots [\ut_{n}}_{\text{$\ell$ times}}, \ut_{\lambda}] \cdots ]]]\quad\text{and}\quad \UT_{\upnum{\lambda}{\ell}} = \underbrace{[\UT_{n}, [\UT_{n}, \cdots [\UT_{n}}_{\text{$\ell$ times}}, \UT_{\lambda}] \cdots ]]],
\]
where $[\cdot, \cdot]$ first denotes the Lie bracket, and then the group commutator.

\subsection{Loopless binary matroids}
\label{sec:matroids}

The construction of ideals of $\ut_{n}(\FF_{q})$ in Section~\ref{sec:ideals} involves indexing certain $\FF_{q}$-vector spaces.  If $q = 2$ this can be done with a certain class of matroid; more generally there is a $q$-analogue for this class, defined below.  This approach is due to the preprint~\cite{HalversonThiem} which is not widely available; for the sake of completeness all statements from this source will be proved.

A \emph{matroid} is a pair $(K, \BBB)$ consisting of a \emph{ground set} $K$ and a collection $\BBB$ of subsets of $K$ called \emph{bases}, satisfying
\begin{enumerate}[label=(M\arabic*),  ref=(M\arabic*)]
\item\label{M1} $\BBB \neq \emptyset$ and

\item\label{M2} for $A, B \in \BBB$ and $a \in K$ with $a \in A$ with $a \notin B$, there is at least one $b \in B$ with $b \notin A$ such that $A \cup \{b\} \setminus \{a\} \in \BBB$.

\end{enumerate}
A \emph{loop} of $(K, \BBB)$ is an element of $K$ which is not contained in any basis of $(K, \BBB)$.  The matroid $(K, \BBB)$ is \emph{loopless} if it has no loops.

A function $\phi: K \to \FF_{q}^{r}$, $r \ge 0$ is an \emph{$\FF_{q}$-representation} of $(K, \BBB)$ if each $A \subseteq K$ satisfies:
\[
A \in \BBB \quad\text{if and only if}\quad \text{$\phi(A)$ is an $\FF_{q}$-basis for $\FF_{q}\operatorname{-span}(\phi(K))$}.
\]
A matroid $(K, \BBB)$ which has an $\FF_{2}$-representation is said to be \emph{binary}. 

\begin{lem}[{\cite[Lemma 3.3]{HalversonThiem}}]
\label{lem:binarymatroidrep}
Let $(K, \BBB)$ be a binary matroid with a totally ordered ground set $K$, and take $A \in \BBB$ to be lexicographically minimal.  Then $(K, \BBB)$ is entirely determined by 
\[
(A; \{E_{b} \;|\; b \notin A\}) \qquad\text{with $E_{b} = \{a \in A \;|\; (A \cup \{b\} \setminus \{a\}) \in \BBB\}$.}
\]
Also, $a < b$ for each $a \in E_{b}$.
\end{lem}

As stated at the beginning of the section, a proof will be included for completeness.

\begin{proof}
Let $\phi: K \to \FF_{2}^{r}$ be a representation of $(K, \BBB)$.  By assumption $\phi(A)$ is a basis for $\FF_{q}\operatorname{-span}(\phi(K))$, so for each $b \notin A$
\begin{equation}\label{eq:binarymatroidrepresentation}
\phi(b) = \sum_{a \in A} \alpha_{a} \phi(a) \qquad\text{for some $\alpha_{a} \in \FF_{2}$}.
\end{equation}
The coefficient $\alpha_{a} = 1$ if and only if $\phi\left(A \cup \{b\} \setminus \{a\}\right)$ is a basis of $\FF_{q}\operatorname{-span}(\phi(K))$, so $E_{b} = \{a \in A \;|\; \alpha_{a} = 1\}$.  Since $A$ is the unique lexicographically minimal element of $\BBB$, this also implies that $a < b$ for each $a \in E_{b}$.  Finally, $\phi(b)$ is determined by $E_{b}$ and $\phi(A)$, so the linear dependence relations on the set $\phi(K)$ can be deduced entirely from $\{E_{b} \;|\; b \notin A\}$.
\end{proof}

In fact, any pair $(A; \{E_{b} \;|\; b \notin A\})$ satisfying $a< b$ for each $a \in E_{b}$ uniquely determines a binary matroid via the $\FF_{2}$-representation in equation~\eqref{eq:binarymatroidrepresentation}.

\subsubsection{A graph model for loopless binary matroids}
\label{sec:matroidgraphs}

A construction of~\cite{HalversonThiem} produces a unique graph from each loopless binary matroid; labeling the edges of these graphs will give the ``labeled loopless binary matroids'' of the main result.  This section gives an exposition of this construction, including proofs for completeness.  A similar construction for arbitrary matroids is also know, see~\cite[Section 6.4]{Oxley}.  

A directed bipartite graph $(U, V, E)$ on the vertex set $U \sqcup V$ with edges $E$ is \emph{unidirectional} if $E \subseteq U \times V$ and for each $v \in V$, there is at least one edge $(u, v) \in E$.  For a set $K$, let
\[
\mathscr{G}_{K} = \{\text{Unidirectional bipartite graphs $(U, V, E)$ with $U \sqcup V = K$}\}.
\]
Now, suppose that $K$ is a totally ordered set, and let
\[
\overset{\to}{\mathscr{G}}_{K} = \{(U, V, E) \in \mathscr{G}_{K} \:|\; \text{$u < v$ for every $(u, v) \in E$}\}.
\]

\begin{ex}\label{ex:unibipartite}
Let $K = \{1 < 2 < 3\}$.  There are six members of $\overset{\to}{\mathscr{G}}_{K}$:
\begin{center}
\begin{tikzpicture}[baseline = 0cm]
\draw[fill] (0, 0.75) circle (2pt) node[inner sep = 2pt] (T1) {};
\draw[fill] (0, 0) circle (2pt) node[inner sep = 2pt] (M) {};
\draw[fill] (0, -0.75) circle (2pt) node[inner sep = 2pt] (T2) {};
\node[left] at (T1) {$1$};
\node[left] at (M) {$2$};
\node[left] at (T2) {$3$};
\end{tikzpicture}\;,
\quad
\begin{tikzpicture}[baseline = 0.5*0.75cm]
\draw[fill] (0, 0.75) circle (2pt) node[inner sep = 2pt] (T1) {};
\draw[fill] (0, 0) circle (2pt) node[inner sep = 2pt] (M) {};
\draw[fill] (1, 0.75) circle (2pt) node[inner sep = 2pt] (T2) {};
\node[left] at (T1) {$1$};
\node[left] at (M) {$2$};
\node[right] at (T2) {$3$};
\draw[->] (T1) -- node[pos = 0.5, above] {} (T2);
\end{tikzpicture},
\quad
\begin{tikzpicture}[baseline = 0.5*0.75cm]
\draw[fill] (0, 0.75) circle (2pt) node[inner sep = 2pt] (T1) {};
\draw[fill] (0, 0) circle (2pt) node[inner sep = 2pt] (M) {};
\draw[fill] (1, 0.75) circle (2pt) node[inner sep = 2pt] (T2) {};
\node[left] at (T1) {$1$};
\node[left] at (M) {$2$};
\node[right] at (T2) {$3$};
\draw[->] (M) -- node[pos = 0.5, above] {} (T2);
\end{tikzpicture},
\quad 
\begin{tikzpicture}[baseline = 0.5*0.75cm]
\draw[fill] (0, 0.75) circle (2pt) node[inner sep = 2pt] (T1) {};
\draw[fill] (0, 0) circle (2pt) node[inner sep = 2pt] (M) {};
\draw[fill] (1, 0.75) circle (2pt) node[inner sep = 2pt] (T2) {};
\node[left] at (T1) {$1$};
\node[left] at (M) {$2$};
\node[right] at (T2) {$3$};
\draw[->] (T1) -- node[pos = 0.5, above] {} (T2);
\draw[->] (M) -- node[pos = 0.2, above] {} (T2);
\end{tikzpicture},
\quad
\begin{tikzpicture}[baseline = 0.5*0.75cm]
\draw[fill] (0, 0.75) circle (2pt) node[inner sep = 2pt] (T1) {};
\draw[fill] (1, 0.75) circle (2pt) node[inner sep = 2pt] (M) {};
\draw[fill] (0, 0) circle (2pt) node[inner sep = 2pt] (T2) {};
\node[left] at (T1) {$1$};
\node[right] at (M) {$2$};
\node[left] at (T2) {$3$};
\draw[->] (T1) -- node[pos = 0.5, above] {} (M);
\end{tikzpicture},\quad\text{and}
\quad
\begin{tikzpicture}[baseline = 0.5*0.75cm]
\draw[fill] (0, 0.75) circle (2pt) node[inner sep = 2pt] (T1) {};
\draw[fill] (1, 0.75) circle (2pt) node[inner sep = 2pt] (M) {};
\draw[fill] (1, 0) circle (2pt) node[inner sep = 2pt] (T2) {};
\node[left] at (T1) {$1$};
\node[right] at (M) {$2$};
\node[right] at (T2) {$3$};
\draw[->] (T1) -- node[pos = 0.5, above] {} (M);
\draw[->] (T1) -- node[pos = 0.8, above] {} (T2);
\end{tikzpicture}.
\end{center}
\end{ex}

\begin{rem}
When $K = \{1 < 2 < \cdots < k \}$, the set $\overset{\to}{\mathscr{G}}_{K}$ is described in~\cite[A135922]{OEIS} under the name ``Stanley graphs,'' though this name does not appear elsewhere in the literature.  
\end{rem}

\begin{prop}[{\cite[Proposition 3.5]{HalversonThiem}}]\label{prop:matroidgraphbijection}
Let $K$ be a totally ordered set.  The maps
\[
\begin{array}{rcl}
\left\{ \text{Loopless binary matroids on $K$}\right\} & \longleftrightarrow & \overset{\to}{\mathscr{G}}_{K} \\[0.5em]
(A; \{E_{v} \;|\; v \notin A\}) & \longmapsto & (A, K \setminus A, \{(u, v) \;|\; v \notin A,\; u \in E_{v}\}) \\[0.35em]
(U; \{E \cap (U \times \{v\}) \;|\; v \in V\})& \longmapsfrom & (U, V, E)
\end{array}
\]
are mutually inverse, giving a bijection.
\end{prop}

\begin{ex}
With $K = \{1 < 2 < 3\}$, Example~\ref{ex:unibipartite} shows the six members of $\overset{\to}{\mathscr{G}}_{K}$.  Under Proposition~\ref{prop:matroidgraphbijection}, the basis of each corresponding loopless binary matroid is, respectively:
\begin{multline*}
\{\{1, 2, 3\}\},\quad\{\{1, 2\}, \{2, 3\}\},\quad \{\{1, 2\}, \{1, 3\}\},\quad\{\{1, 2\}, \{1, 3\}, \{2, 3\}\},\\ \quad\{\{1, 3\}, \{2, 3\}\},\; \text{and} \quad \{\{1\}, \{2\}, \{3\}\}.
\end{multline*}
\end{ex}

An $\FF_{q}^{\times}$-labeling for an element $(U, V, E) \in \overset{\to}{\mathscr{G}}_{K}$ is a function $\tau: E \to \FF_{q}^{\times}$.  In light of Proposition~\ref{prop:matroidgraphbijection}, define the set of \emph{$\FF_{q}^{\times}$-labeled loopless binary matroids} on $K$ to be
\[
\overset{\to}{\mathscr{G}}_{K}(q) = \{(\MMM, \tau) \;|\; \MMM = (U, V, E) \in \overset{\to}{\mathscr{G}}_{K},\; \tau: E \to \FF_{q}^{\times}\}.
\]
Draw elements of $\overset{\to}{\mathscr{G}}_{K}(q)$ as edge-labeled graphs.

\begin{ex}
Taking $K =\{1 < 2 < 3\}$ as in the previous example, several elements of $\overset{\to}{\mathscr{G}}_{K}(5)$ are shown below.
\begin{align*}
\text{With}\;\MMM = \begin{tikzpicture}[baseline = 0.5*0.75cm]
\draw[fill] (0, 0.75) circle (2pt) node[inner sep = 2pt] (T1) {};
\draw[fill] (0, 0) circle (2pt) node[inner sep = 2pt] (M) {};
\draw[fill] (1, 0.75) circle (2pt) node[inner sep = 2pt] (T2) {};
\node[left] at (T1) {$1$};
\node[left] at (M) {$2$};
\node[right] at (T2) {$3$};
\draw[->] (M) -- node[pos = 0.6, below] {} (T2);
\end{tikzpicture}
\;\;\text{and $\tau(2, 3) = 4$:}& \qquad
(\MMM, \tau) = \begin{tikzpicture}[baseline = 0.5*0.75cm]
\draw[fill] (0, 0.75) circle (2pt) node[inner sep = 2pt] (T1) {};
\draw[fill] (0, 0) circle (2pt) node[inner sep = 2pt] (M) {};
\draw[fill] (1, 0.75) circle (2pt) node[inner sep = 2pt] (T2) {};
\node[left] at (T1) {$1$};
\node[left] at (M) {$2$};
\node[right] at (T2) {$3$};
\draw[->] (M) -- node[pos = 0.6, below] {\color{red} $4$} (T2);
\end{tikzpicture}. 
\\[0.25em]
\text{With}\;\MMM = \begin{tikzpicture}[baseline = 0.5*0.75cm]
\draw[fill] (0, 0.75) circle (2pt) node[inner sep = 2pt] (T1) {};
\draw[fill] (0, 0) circle (2pt) node[inner sep = 2pt] (M) {};
\draw[fill] (1, 0.75) circle (2pt) node[inner sep = 2pt] (T2) {};
\node[left] at (T1) {$1$};
\node[left] at (M) {$2$};
\node[right] at (T2) {$3$};
\draw[->] (T1) -- node[pos = 0.5, above] {} (T2);
\draw[->] (M) -- node[pos = 0.6, below] {} (T2);
\end{tikzpicture}
,\;\;\text{$\tau(1, 3) = 1$, and $\tau(2, 3) = 2$:}&\qquad
(\MMM, \tau) = \begin{tikzpicture}[baseline = 0.5*0.75cm]
\draw[fill] (0, 0.75) circle (2pt) node[inner sep = 2pt] (T1) {};
\draw[fill] (0, 0) circle (2pt) node[inner sep = 2pt] (M) {};
\draw[fill] (1, 0.75) circle (2pt) node[inner sep = 2pt] (T2) {};
\node[left] at (T1) {$1$};
\node[left] at (M) {$2$};
\node[right] at (T2) {$3$};
\draw[->] (T1) -- node[pos = 0.5, above] {\color{red} $1$} (T2);
\draw[->] (M) -- node[pos = 0.6, below] {\color{red} $2$} (T2);
\end{tikzpicture}. 
\\[0.25em]
\text{With}\;\MMM = \begin{tikzpicture}[baseline = 0.5*0.75cm]
\draw[fill] (0, 0.75) circle (2pt) node[inner sep = 2pt] (T1) {};
\draw[fill] (1, 0.75) circle (2pt) node[inner sep = 2pt] (M) {};
\draw[fill] (0, 0) circle (2pt) node[inner sep = 2pt] (T2) {};
\node[left] at (T1) {$1$};
\node[right] at (M) {$2$};
\node[left] at (T2) {$3$};
\draw[->] (T1) -- node[pos = 0.5, above] {} (M);
\end{tikzpicture}
\;\;\text{and $\tau(1, 2) = 3$:} & \qquad
(\MMM, \tau) = \begin{tikzpicture}[baseline = 0.5*0.75cm]
\draw[fill] (0, 0.75) circle (2pt) node[inner sep = 2pt] (T1) {};
\draw[fill] (1, 0.75) circle (2pt) node[inner sep = 2pt] (M) {};
\draw[fill] (0, 0) circle (2pt) node[inner sep = 2pt] (T2) {};
\node[left] at (T1) {$1$};
\node[right] at (M) {$2$};
\node[left] at (T2) {$3$};
\draw[->] (T1) -- node[pos = 0.5, above] {\color{red} $3$} (M);
\end{tikzpicture}.
\end{align*}
\end{ex}

The \emph{$q$-Stirling numbers of the second kind} are defined by
\begin{equation}
\label{eq:qstirling}
\qstirling{k}{j}_{q} =
\sum_{\substack{V \subseteq [k] \\ |V| = k - j}} \prod_{v \in V} \frac{q^{i_{V}(v)} - 1}{q-1}
\qquad \text{with} \qquad \operatorname{i}_{V}(v) = |\{w \in [k] \setminus V \;|\; w < v\}|
\end{equation}
for $k \ge j \ge 0$.  Note that subsets of $[k]$ which contain $1$ contribute nothing to the sum; also note that $\qstirling{k}{j}_{q}$ is a polynomial expression in $q-1$ with positive coefficients, as
\[
\frac{q^{s} - 1}{q-1} = \sum_{i = 0}^{s-1} q^{i} = \sum_{i = 0}^{s-1} \binom{s}{i+1} (q-1)^{i} \qquad \text{for $s \ge 0$.}
\]

\begin{rem}
Formula~\eqref{eq:qstirling} is due to~\cite{Gould} (see~\cite[{Equation 10.1}]{CaiReaddy} for a recent statement) and is a specialization of the symmetric function $h_{k - j}$ of~\cite[Section I.2]{Macdonald}.  Other definitions, including a Stirling-like recurrence, can be found in~\cite{CaiReaddy} and~\cite{WachsWhite}, among other sources.
\end{rem}

\begin{prop}[{\cite[Corollary 3.4]{HalversonThiem}}]
\label{prop:matroidcount}
For a totally ordered set $K$ with $k$ elements,
\[
\big| \overset{\to}{\mathscr{G}}_{K}(q) \big| = \sum_{j = 0}^{k} (q-1)^{k-j} \qstirling{k}{j}_{q}.
\]
\end{prop}
\begin{proof}
Without loss of generality, take $K = \{1 < 2 < \cdots < k\}$.  It is sufficient to show that
\[
\big| \overset{\to}{\mathscr{G}}_{K}(q) \big| = \sum_{V \subseteq [k] \setminus \{1\}} \prod_{v \in V} (q^{\operatorname{i}_{V}(v)} - 1).
\]

Compute $\big| \overset{\to}{\mathscr{G}}_{K}(q) \big|$ by summing over choices of $V \subseteq K$, which must not contain $1$.  Fixing $V$ and taking $U = K \setminus V$, each choice of $E$ and $\tau$ is equivalent to a function $\tilde{\tau}: U \times V \to \FF_{q}$ with
\[
\tilde{\tau}(u, v) = \begin{cases} \tau(u, v) & \text{if $(u, v) \in E$,} \\ 0 & \text{otherwise.} \end{cases}
\]
For each $v \in V$, there are $q^{i_{V}(v)} - 1$ allowable restrictions of $\tilde{\tau}$ to $U \times \{v\}$.  Finally, $\tilde{\tau}$ is determined by these restricted functions, each of which may be chosen independently.
\end{proof}


\section{Splices}
\label{sec:splices}

It is not obvious, but for a given ideal $\mathfrak{n} \subseteq \ut_{n}(\FF_{q})$ there are relatively few $\displaystyle \edge{i}{j} \in \operatorname{supp}(\mathfrak{n})$ for which $\FF_{q}e_{i, j} \not\subseteq \mathfrak{n}$.  These edges occur according to certain straightforward rules, which are formalized in the definition of a \emph{tight splice of a nonnesting set partition}.  Studying a generalization of these rules first leads to a few interesting connections.

Let $\lambda$ and $\nu$ be disjoint set partitions of $[n]$.  The set $\SSS = \lambda \sqcup \nu$ is a \emph{splice} of $\lambda$ if:
\begin{enumerate}[label=(S\arabic*),  ref=(S\arabic*)]

\item \label{S1} for each $\displaystyle \dedge{i}{k} \in \nu$, $\lambda$ contains at least one edge $\displaystyle \edge{i}{j}$ or $\displaystyle \edge{j}{k}$ with $i < j < k$; and

\item \label{S2} for $1 \le j < k \le n$, $\displaystyle \dedge{i}{k} \in \nu$ and $\displaystyle \edge{i}{j} \in \lambda$  if and only if $\displaystyle \edge{k}{l} \in \lambda$ and $\displaystyle \dedge{j}{l} \in \nu$. 

\end{enumerate}
Every set partition $\lambda$ has at least one splice, the trivial splice $\lambda = \lambda \sqcup \emptyset$.  

Say that $\lambda$ is the \emph{underlying partition} of the splice $\SSS = \lambda \sqcup \nu$, and that $\nu$ is the set of \emph{vertical edges} of $\SSS$.  Both sets are determined by $\SSS$: $\nu = \SSS \setminus \lambda$, and 
\[
\lambda = \Big\{\edge{i}{j} \in \SSS \;\Big|\; \text{no element of $\SSS$ precedes $\edge{i}{j}$ in either $\preceq_{L}$ or $\preceq_{R}$}\Big\}.
\]
In Sections~\ref{sec:ideals} and~\ref{sec:normalsubgroups}, $\lambda$ will always be nonnesting.  In this case, the above equation simplifies to $\lambda = \min(\SSS)$,  the $\preceq$-minimal elements of $\SSS$.  

When drawing a splice $\SSS = \lambda \sqcup \nu$, the convention will be to draw edges from $\lambda$ horizontally and $\nu$ vertically, as in~\ref{S1} and~\ref{S2}.  For example, 
\begin{equation}\label{eq:spliceex1}
\SSS = 
\begin{tikzpicture}[baseline = -0.1cm]
\begin{scope}[xshift = 0cm, yshift = 1cm]
\draw[fill] (0, 0) circle (2pt) node[inner sep = 2pt] (1) {};
\draw[fill] (1, 0) circle (2pt) node[inner sep = 2pt] (2) {};
\end{scope}
\begin{scope}[xshift = 3cm, yshift = 0cm]
\draw[fill] (0, 0) circle (2pt) node[inner sep = 2pt] (3) {};
\draw[fill] (1, 0) circle (2pt) node[inner sep = 2pt] (5) {};
\draw[fill] (2, 0) circle (2pt) node[inner sep = 2pt] (7) {};
\end{scope}
\begin{scope}[xshift = 0cm, yshift = 0cm]
\draw[fill] (0, 0) circle (2pt) node[inner sep = 2pt] (4) {};
\draw[fill] (1, 0) circle (2pt) node[inner sep = 2pt] (6) {};
\draw[fill] (2, 0) circle (2pt) node[inner sep = 2pt] (8) {};
\end{scope}
\begin{scope}[xshift = 0cm, yshift = -1cm]
\draw[fill] (0, 0) circle (2pt) node[inner sep = 2pt] (9) {};
\draw[fill] (1, 0) circle (2pt) node[inner sep = 2pt] (10) {};
\draw[fill] (2, 0) circle (2pt) node[inner sep = 2pt] (11) {};
\end{scope}
\node[above] at (1) {$\scriptstyle 1$};
\node[above] at (2) {$\scriptstyle 2$};
\node[below] at (3) {$\scriptstyle 3$};
\node[above left] at (4) {$\scriptstyle 4$};
\node[below] at (5) {$\scriptstyle 5$};
\node[above right] at (6) {$\scriptstyle 6$};
\node[below] at (7) {$\scriptstyle 7$};
\node[above] at (8) {$\scriptstyle 8$};
\node[below] at (9) {$\scriptstyle 9$};
\node[below] at (10) {$\scriptstyle 10$};
\node[below] at (11) {$\scriptstyle 11$};
\draw[thick] (1) -- (2);
\draw[thick] (3) -- (5);
\draw[thick] (5) -- (7);
\draw[thick] (4) -- (6);
\draw[thick] (6) -- (8);
\draw[thick] (9) -- (10);
\draw[thick] (10) -- (11);
\draw[thick, dashed, color=specialgold] (1) -- (4);
\draw[thick, dashed, color=specialgold] (2) -- (6);
\draw[thick, dashed, color=specialgold] (4) -- (9);
\draw[thick, dashed, color=specialgold] (6) -- (10);
\draw[thick, dashed, color=specialgold] (8) -- (11);
\end{tikzpicture}
\quad\quad\text{ with $\displaystyle \edge{}{} \in \lambda$ and $\displaystyle \dedge{}{} \in \nu$}
\end{equation}
shows a splice $\SSS = \lambda \sqcup \nu$ of $\lambda = \{\edge{1}{2}, \edge{3}{5}, \edge{4}{6}, \edge{5}{7}, \edge{6}{8}, \edge{9}{10}, \edge{10}{11}\} \in \mathrm{NNSP}_{11}$, with  $\nu = \left\{ \dedge{1}{4}, \dedge{2}{6}, \dedge{4}{9}, \dedge{6}{10}, \dedge{8}{11} \right\}$.

\begin{rems}
\begin{enumerate}[label=(R\arabic*),  ref=(R\arabic*)]
\item In his groundbreaking work on the representation theory of $\UT_{n}(\FF_{q})$~\cite{Andre2, Andre1}, Andr\'{e} defines a set of ``superclasses'' which are unions of conjugacy classes.  Later work~\cite{AguiarEtAl} indexes these classes with labeled set partitions of $[n]$.  

In the language of splices,~\cite[Theorem 2.2]{Andre1} states that a superclass indexed by the set partition $\lambda$ is a single conjugacy class if and only if $\lambda$ has no nontrivial splices.  This suggests that splices may be useful in further study of Andr\'{e}'s superclasses.

\item \label{R2} If $\SSS$ is a splice of a nonnesting set partition, each connected component in the graph of a $\SSS$ can be drawn in a grid-like shape, as in~\eqref{eq:spliceex1}.  This shape coincides with that of a \emph{shortened polyomino}~\cite{SzuEnEtAl07, SzuEnEtAl13, DeutschShapiro}, defined as a pair of lattice paths in $\ZZ^{2}$ subject to certain conditions.  For example, the splice in~\eqref{eq:spliceex1} corresponds to the shortened polyominoes
\[
\begin{tikzpicture}[baseline = -1.1cm]
\draw[thin, gray] (-0.25, 0.25) grid (2.25, -2.25);
\draw[fill] (0, 0) circle (2pt); 
\draw[fill] (1, 0) circle (2pt); 
\draw[fill] (0, -1) circle (2pt); 
\draw[fill] (1, -1) circle (2pt); 
\draw[fill] (2, -1) circle (2pt); 
\draw[fill] (0, -2) circle (2pt); 
\draw[fill] (1, -2) circle (2pt); 
\draw[fill] (2, -2) circle (2pt); 
\draw[thick] (0, 0) -- (1, 0) -- (1, -1) -- (2, -1) -- (2, -2);
\draw[thick] (0, 0) -- (0, -1) -- (0, -2) -- (1, -2) -- (2, -2);
\end{tikzpicture}
\;=\; (ESES, SSEE)
\qquad\text{and}\qquad
\begin{tikzpicture}[baseline = -1.1cm]
\draw[thin, gray] (-0.25, 0.25) grid (2.25, -2.25);
\draw[fill] (0, 0) circle (2pt); 
\draw[fill] (1, 0) circle (2pt); 
\draw[fill] (2, 0) circle (2pt); 
\draw[thick] (0, 0) -- (2, 0);
\end{tikzpicture}
\;=\; (EE, EE).
\]
The number of shortened polyominoes with $k$ steps in each path is the $k$th Catalan number.  This connection has yet to be thoroughly investigated, but it does seem to be significant: each important statistic on splices matches a known statistic for shortened polyominoes, and thus other Catalan objects.
\end{enumerate}
\end{rems}

\subsection{Bindings, rows, and columns}
\label{sec:bindingsandcolumns}

Given a splice $\SSS = \lambda \sqcup \nu$, define the \emph{bindings} of $\SSS$ to be elements of
\[
\operatorname{bind}(\SSS) = \left\{ \binding{i}{j}{k}{l} \;\middle|\;  \edge{i}{j}, \edge{k}{l}\in \lambda\;\text{and} \dedge{i}{k}\;,\dedge{j}{l} \in \nu\;\text{with}\;j < k \right\},
\]
so that $\operatorname{bind}(\SSS)$ records the tuples of edges to which~\ref{S2} applies.  

If $\SSS = \lambda \sqcup \nu$ is a splice, then $\nu$ gives an equivalence relation $\sim_{\text{cols}}$ on $\lambda$, generated by 
\[
\edge{i}{j} \sim_{\text{cols}} \edge{k}{l} \quad \text{if} \quad \binding{i}{j}{k}{l} \in \operatorname{bind}(\SSS).
\]
Define the \emph{columns} of $\SSS$ to be the equivalence classes of $\lambda$ under $\sim_{\text{cols}}$, and write
\[
\operatorname{cols}(\SSS) = \{\CCC \subseteq \lambda \;|\; \text{$\CCC$ is a column of $\SSS$}\}.
\]
Each $\CCC \in \operatorname{cols}(\SSS)$ has the form $\CCC = \{\edge{i_{\ell}}{j_{\ell}} \;|\; 1 \le \ell \le |\CCC|\}$ for
\begin{center}
\begin{tikzpicture}[baseline = -1.5cm]
\begin{scope}[xshift = 0cm, yshift = 0cm]
\draw[fill] (0, 0) circle (2pt) node[inner sep=2pt] (i1) {};
\draw[fill] (0, -1) circle (2pt) node[inner sep=2pt] (i2) {};
\draw[fill] (0, -2.25) circle (2pt) node[inner sep=2pt] (icm1) {};
\draw[fill] (0, -3.25) circle (2pt) node[inner sep=2pt] (ic) {};
\node[specialgold] at (0, -1.525) (idots) {$\vdots$};
\path (i2) -- node[midway, inner sep = 9pt] (phantomidots) {} (icm1);
\end{scope}
\begin{scope}[xshift = 1cm, yshift = 0cm]
\draw[fill] (0, 0) circle (2pt) node[inner sep=2pt] (j1) {};
\draw[fill] (0, -1) circle (2pt) node[inner sep=2pt] (j2) {};
\draw[fill] (0, -2.25) circle (2pt) node[inner sep=2pt] (jcm1) {};
\draw[fill] (0, -3.25) circle (2pt) node[inner sep=2pt] (jc) {};
\node[color = specialgold] at (0, -1.525) (jdots) {$\vdots$};
\path (j2) -- node[midway, inner sep = 9pt] (phantomjdots) {} (jcm1);
\end{scope}
\node[left] at (i1) {$\scriptstyle i_{1}$};
\node[left] at (i2) {$\scriptstyle i_{2}$};
\node[left] at (icm1) {$\scriptstyle i_{|\CCC| - 1}$};
\node[left] at (ic) {$\scriptstyle i_{|\CCC|}$};
\node[right] at (j1) {$\scriptstyle j_{1}$};
\node[right] at (j2) {$\scriptstyle j_{2}$};
\node[right] at (jcm1) {$\scriptstyle j_{|\CCC| - 1}$};
\node[right] at (jc) {$\scriptstyle j_{|\CCC|}$};
\draw[thick] (i1) -- (j1);
\draw[thick] (i2) -- (j2);
\path (idots) -- node[midway] {$\vdots$} (jdots);
\draw[thick] (icm1) -- (jcm1);
\draw[thick] (ic) -- (jc);
\draw[thick, dashed, color=specialgold] (i1) -- (i2);
\draw[thick, dashed, color=specialgold] (j1) -- (j2);
\draw[thick, dashed, color=specialgold] (i2) -- (phantomidots);
\draw[thick, dashed, color=specialgold] (j2) -- (phantomjdots);
\draw[thick, dashed, color=specialgold] (phantomidots) -- (icm1);
\draw[thick, dashed, color=specialgold] (phantomjdots) -- (jcm1);
\draw[thick, dashed, color=specialgold] (icm1) -- (ic);
\draw[thick, dashed, color=specialgold] (jcm1) -- (jc);
\end{tikzpicture}
$\subseteq \SSS$
\qquad\text{with $\displaystyle \edge{}{} \in \lambda$ and $\displaystyle \dedge{}{} \in \nu$,}
\end{center}
including the noteworthy special case of $\{\edge{i}{j}\} \in \operatorname{cols}(\SSS)$ when $\edge{i}{j}$ is not contained in any bindings of $\SSS$.  As a result,
\begin{equation}\label{eq:colsformula}
|\lambda| = |\operatorname{cols}(\SSS)| + |\operatorname{bind}(\SSS)|.
\end{equation}

Similarly, $\lambda$ gives an equivalence relation $\sim_{\text{rows}}$ on $\nu$, generated by
\[
\dedge{i}{k}\; \sim_{\text{rows}} \dedge{j}{l} \quad \text{if} \quad \binding{i}{j}{k}{l} \in \operatorname{bind}(\SSS).
\]
Let the \emph{rows} of $\SSS$ be the equivalence classes of $\nu$ under $\sim_{\text{rows}}$, and write
\[
\operatorname{rows}(\SSS) = \{\RRR \subseteq \nu \;|\; \text{$\RRR$ is a row of $\SSS$}\}.
\]
Each $\RRR \in \operatorname{rows}(\SSS)$ is of the form $\RRR = \left\{ \dedge[color = specialgold, dashed]{i_{\ell}}{k_{\ell}} \;\middle|\; 1 \le \ell \le |\RRR| \right\}$ for
\begin{center}
\begin{tikzpicture}[baseline = -0.6cm]
\begin{scope}[xshift = 0cm, yshift = -1cm]
\draw[fill] (0, 0) circle (2pt) node[inner sep = 2pt] (k0) {};
\draw[fill] (1, 0) circle (2pt) node[inner sep = 2pt] (k1) {};
\node[inner sep = 5pt] (kdots) at (1.75, 0) {$\cdots$};
\draw[fill] (2.5, 0) circle (2pt) node[inner sep = 2pt] (kv) {};
\draw[fill] (3.5, 0) circle (2pt) node[inner sep = 2pt] (kr) {};
\end{scope}
\begin{scope}[xshift = 0cm, yshift = 0cm]
\draw[fill] (0, 0) circle (2pt) node[inner sep = 2pt] (i0) {};
\draw[fill] (1, 0) circle (2pt) node[inner sep = 2pt] (i1) {};
\node[inner sep = 5pt] (idots) at (1.75, 0) {$\cdots$};
\draw[fill] (2.5, 0) circle (2pt) node[inner sep = 2pt] (iv) {};
\draw[fill] (3.5, 0) circle (2pt) node[inner sep = 2pt] (ir) {};
\end{scope}
\node[above] at (i0) {$\scriptstyle i_{1}$};
\node[above] at (i1) {$\scriptstyle i_{2}$};
\node[above] at (iv) {$\scriptstyle i_{|\RRR| - 1}$};
\node[above] at (ir) {$\scriptstyle i_{|\RRR|}$};
\node[below] at (k0) {$\scriptstyle k_{1}$};
\node[below] at (k1) {$\scriptstyle k_{2}$};
\node[below] at (kv) {$\scriptstyle k_{|\RRR| - 1}$};
\node[below] at (kr) {$\scriptstyle k_{|\RRR| }$};
\draw[thick, dashed, color=specialgold] (i0) -- (k0);
\draw[thick, dashed, color=specialgold] (i1) -- (k1);
\path[color=specialgold] (idots) -- node[midway] {$\cdots$} (kdots);
\draw[thick, dashed, color=specialgold] (iv) -- (kv);
\draw[thick, dashed, color=specialgold] (ir) -- (kr);
\draw[thick] (i0) -- (i1);
\draw[thick] (i1) -- (idots);
\draw[thick] (idots) -- (iv);
\draw[thick] (iv) -- (ir);
\draw[thick] (k0) -- (k1);
\draw[thick] (k1) -- (kdots);
\draw[thick] (kdots) -- (kv);
\draw[thick] (kv) -- (kr);
\end{tikzpicture}
$\subseteq \SSS$
\qquad with $\displaystyle \edge{}{} \in \lambda$ and $\displaystyle \dedge[dashed, color = specialgold]{}{} \in \nu$.
\end{center}
This gives a row counterpart to equation~\eqref{eq:colsformula},
\begin{equation}
\label{eq:rowsformula}
|\nu| = |\operatorname{rows}(\SSS)| + |\operatorname{bind}(\SSS)|.
\end{equation}
Unlike columns,~\ref{S1} and~\ref{S2} ensure that every row has at least two elements.

For the splice $\SSS = \lambda \sqcup \nu$ with $\lambda = \{\edge{1}{2}, \edge{3}{5}, \edge{4}{6}, \edge{5}{7}, \edge{6}{8}, \edge{9}{10}, \edge{10}{11}\}$ and $\nu = \left\{ \dedge{1}{4}, \dedge{2}{6}, \dedge{4}{9}, \dedge{6}{10}, \dedge{8}{11} \right\}$ as in~\eqref{eq:spliceex1}, 
\[
\operatorname{cols}(\SSS) = \Big\{\Big\{\edge{1}{2}, \edge{4}{6}, \edge{9}{10} \Big\},  \Big\{ \edge{6}{8}, \edge{10}{11} \Big\}, \Big\{ \edge{3}{5} \Big\}, \Big\{\edge{5}{7} \Big\}\Big\}.
\]
and 
\[
\operatorname{rows}(\SSS) = \left\{ \left\{ \dedge{1}{4}\;, \dedge{2}{6} \right\}, \left\{ \dedge{4}{9}\;, \dedge{6}{10}\;, \dedge{8}{11} \right\}\right\}.
\]

\begin{lem}\label{lem:spliceproperties}
Let $\SSS = \lambda \sqcup \nu$ be a splice and $\RRR$ a row of $\SSS$.  Let $I$ and $K$ be the connected components in the graph of $\lambda$ which contain $\left\{ i \;\middle|\; \dedge{i}{k} \in \RRR \right\}$ and $\left\{ k \;\middle|\; \dedge{i}{k} \in \RRR \right\}$, respectively.  Then $I \neq K$, and if $\lambda$ is nonnesting, every $\dedge{i}{k} \in \nu$ with $i \in I$ or $k \in K$ belongs to $\RRR$.
\end{lem}
\begin{proof}
Let $s = \max(\left\{ i \;\middle|\; \dedge{i}{k} \in \RRR \right\})$.  By~\ref{S1} and~\ref{S2} $\SSS$ has a binding 
\[
\binding{r}{s}{t}{u} \in \operatorname{bind}(\SSS)
\qquad \text{with $\edge{r}{s}, \edge{t}{u} \in \lambda$ and $\dedge[dashed, color = specialgold]{r}{t}\;, \dedge[dashed, color = specialgold]{s}{u} \in \RRR \subseteq \nu$},
\]
and either $s = \max(I)$, or $\edge{s}{v} \in \lambda$ with $v > u$.  In either case, $u \notin I$, so $I \neq K$.

If $\lambda$ nonnesting, only $s = \max(I)$ can be true because $\edge{t}{u} \prec \edge{s}{v}$.  Thus if $\edge{i}{k} \in \nu$ with $i \in I$, repeated application of~\ref{S2} shows that $\dedge{i}{k} \sim_{\text{rows}} \dedge{s}{w}$ for some $\dedge{s}{w} \in \nu$.  As $\nu$ is a set partition, $w = u$ and $\dedge{i}{k} \in \RRR$.  A similar line of reasoning applied to $\min(\left\{ k \;\middle|\; \dedge{i}{k} \in \RRR \right\})$ shows that $\dedge{i}{k} \in \nu$ with $k \in K$ also shows that $\dedge{i}{k} \in \RRR$.
\end{proof}

With the description of rows given above, Lemma~\ref{lem:spliceproperties} can be used to show that the graph of a splice of a nonnesting set partition is planar and has the properties described in~\ref{R2}.  However, these facts will not be used in the scope of this paper.

\subsection{Tight splices}
\label{sec:shortsplices}

Take $\lambda \in \mathrm{NNSP}_{n}$, so that $\lambda$ is a nonnesting set partition.  Say that a splice $\SSS$ of $\lambda$ is \emph{tight} if $\SSS \cap \upnum{\lambda}{2} = \emptyset$.  By ~\ref{S2}, this is equivalent to the more straightforward condition
\[
\tag{T}\label{T} \binding{i}{j}{k}{l} \in \operatorname{bind}(\SSS)\;\text{ only if $k = j+1$.}
\]

\begin{prop}
Let $\SSS = \lambda \sqcup \nu$ be a tight splice.  If $\lambda$ is nonnesting, then so is $\nu$.
\end{prop}
\begin{proof}
As $\lambda$ and $\nu$ are disjoint, tightness implies that $\nu \subseteq \upnum{\lambda}{1} \setminus \upnum{\lambda}{2}$.  Now observe that $\upnum{\lambda}{1} \setminus \upnum{\lambda}{2}$ is an antichain under $\prec$.
\end{proof}

The splice of the nonnesting set partition $\lambda = \{\edge{1}{2}, \edge{3}{5}, \edge{4}{6}, \edge{5}{7}, \edge{6}{8}, \edge{9}{10}, \edge{10}{11}\}$ in~\eqref{eq:spliceex1} is not tight, but
\begin{equation}\label{eq:spliceex2}
S = 
\begin{tikzpicture}[baseline = -0.6cm]
\begin{scope}[xshift = 0cm, yshift = 0cm]
\draw[fill] (0, 0) circle (2pt) node[inner sep = 2pt] (1) {};
\draw[fill] (1, 0) circle (2pt) node[inner sep = 2pt] (2) {};
\end{scope}
\begin{scope}[xshift = 0cm, yshift = -1cm]
\draw[fill] (0, 0) circle (2pt) node[inner sep = 2pt] (3) {};
\draw[fill] (1, 0) circle (2pt) node[inner sep = 2pt] (5) {};
\draw[fill] (2, 0) circle (2pt) node[inner sep = 2pt] (7) {};
\end{scope}
\begin{scope}[xshift = 3cm, yshift = 0cm]
\draw[fill] (0, 0) circle (2pt) node[inner sep = 2pt] (4) {};
\draw[fill] (1, 0) circle (2pt) node[inner sep = 2pt] (6) {};
\draw[fill] (2, 0) circle (2pt) node[inner sep = 2pt] (8) {};
\end{scope}
\begin{scope}[xshift = 4cm, yshift = -1cm]
\draw[fill] (0, 0) circle (2pt) node[inner sep = 2pt] (9) {};
\draw[fill] (1, 0) circle (2pt) node[inner sep = 2pt] (10) {};
\draw[fill] (2, 0) circle (2pt) node[inner sep = 2pt] (11) {};
\end{scope}
\node[above] at (1) {$\scriptstyle 1$};
\node[above] at (2) {$\scriptstyle 2$};
\node[below] at (3) {$\scriptstyle 3$};
\node[above] at (4) {$\scriptstyle 4$};
\node[below] at (5) {$\scriptstyle 5$};
\node[above] at (6) {$\scriptstyle 6$};
\node[below] at (7) {$\scriptstyle 7$};
\node[above] at (8) {$\scriptstyle 8$};
\node[below] at (9) {$\scriptstyle 9$};
\node[below] at (10) {$\scriptstyle 10$};
\node[below] at (11) {$\scriptstyle 11$};
\draw[thick] (1) -- (2);
\draw[thick] (3) -- (5);
\draw[thick] (5) -- (7);
\draw[thick] (4) -- (6);
\draw[thick] (6) -- (8);
\draw[thick] (9) -- (10);
\draw[thick] (10) -- (11);
\draw[thick, dashed, color=specialgold] (1) -- (3);
\draw[thick, dashed, color=specialgold] (2) -- (5);
\draw[thick, dashed, color=specialgold] (6) -- (9);
\draw[thick, dashed, color=specialgold] (8) -- (10);
\end{tikzpicture}
\end{equation}
is a tight splice of $\lambda$, with vertical edges $\{\dedge{1}{3}, \dedge{2}{5}, \dedge{6}{9}, \dedge{8}{10}\}$.

For $\lambda \in \mathrm{NNSP}_{n}$, let
\[
\mathscr{T}_{\lambda} = \{\text{tight splices $\SSS$ of $\lambda$}\}
\]
and
\[
\mathscr{T}_{n} = \bigsqcup_{\lambda \in \mathrm{NNSP}_{n}} \mathscr{T}_{\lambda}.
\]

\begin{prop}\label{prop:completesplice}
Let $\lambda \in \mathrm{NNSP}_{n}$.  The set $\mathscr{T}_{\lambda}$ has a maximum element $\KKK$, and 
\[
\mathscr{T}_{\lambda} = \left\{ \lambda \sqcup \nu \;\middle|\; \text{$\nu$ is a union of elements of $\operatorname{rows}(\KKK)$} \right\}.
\]
\end{prop}
\begin{proof}
For $\SSS \in \mathscr{T}_{\lambda}$ and $\RRR \in \operatorname{rows}(\SSS)$, the set $\lambda \sqcup \RRR \in \mathscr{T}_{\lambda}$ is a tight splice of $\lambda$.  It is therefore sufficient to show that the union of any subset of $\mathscr{T}_{\lambda}$ is a tight splice of $\lambda$.

Let $\SSS_{1} = \lambda \sqcup \nu_{1}, \SSS_{2} = \lambda \sqcup \nu_{2}, \ldots, \SSS_{\ell} = \lambda \sqcup \nu_{\ell}$ be tight splices of $\lambda$, and let
\[
\TTT = \SSS_{1} \cup \SSS_{2} \cup \cdots \cup \SSS_{\ell} = \lambda \sqcup \nu \qquad\text{with $\nu =  \bigcup_{i = 1}^{\ell} \nu_{i}$}.
\]
Then $\nu \subseteq \upnum{\lambda}{1} \setminus \hspace{-0.25cm} \upnum{\lambda}{2}$, so $\nu$ is a nonnesting set partition.  Lastly, each $\SSS_{i}$ satisfies~\ref{S1} and~\ref{S2}, which are existentially quantified, so $\TTT$ does as well.
\end{proof}

\subsection{Labeling tight splices}
\label{sec:labelingsplices}

For $\SSS \in \mathscr{T}_{n}$, an $\FF_{q}^{\times}$-labeling of $\SSS$ will be a function $\sigma: \operatorname{bind}(\SSS) \to \FF_{q}^{\times}$.  Let
\[
\mathscr{T}_{\lambda}(q) = \{(\SSS, \sigma) \;|\; \SSS \in \mathscr{T}_{\lambda},\; \sigma: \operatorname{bind}(\SSS) \to \FF_{q}^{\times}\}
\]
and
\[
\mathscr{T}_{n}(q) = \bigsqcup_{\lambda \in \mathrm{NNSP}_{n}} \mathscr{T}_{\lambda}(q).
\]
Such a labeling can be realized graphically by drawing label values in the center of each binding.  For example, with $p > 3$,
\begin{center}
\begin{tikzpicture}[baseline = -0.6cm]
\begin{scope}[xshift = 0cm, yshift = 0cm]
\draw[fill] (0, 0) circle (2pt) node[inner sep = 2pt] (1) {};
\draw[fill] (1, 0) circle (2pt) node[inner sep = 2pt] (2) {};
\end{scope}
\begin{scope}[xshift = 0cm, yshift = -1cm]
\draw[fill] (0, 0) circle (2pt) node[inner sep = 2pt] (3) {};
\draw[fill] (1, 0) circle (2pt) node[inner sep = 2pt] (5) {};
\draw[fill] (2, 0) circle (2pt) node[inner sep = 2pt] (7) {};
\end{scope}
\begin{scope}[xshift = 3cm, yshift = 0cm]
\draw[fill] (0, 0) circle (2pt) node[inner sep = 2pt] (4) {};
\draw[fill] (1, 0) circle (2pt) node[inner sep = 2pt] (6) {};
\draw[fill] (2, 0) circle (2pt) node[inner sep = 2pt] (8) {};
\end{scope}
\begin{scope}[xshift = 4cm, yshift = -1cm]
\draw[fill] (0, 0) circle (2pt) node[inner sep = 2pt] (9) {};
\draw[fill] (1, 0) circle (2pt) node[inner sep = 2pt] (10) {};
\draw[fill] (2, 0) circle (2pt) node[inner sep = 2pt] (11) {};
\end{scope}
\node[above] at (1) {$\scriptstyle 1$};
\node[above] at (2) {$\scriptstyle 2$};
\node[below] at (3) {$\scriptstyle 3$};
\node[above] at (4) {$\scriptstyle 4$};
\node[below] at (5) {$\scriptstyle 5$};
\node[above] at (6) {$\scriptstyle 6$};
\node[below] at (7) {$\scriptstyle 7$};
\node[above] at (8) {$\scriptstyle 8$};
\node[below] at (9) {$\scriptstyle 9$};
\node[below] at (10) {$\scriptstyle 10$};
\node[below] at (11) {$\scriptstyle 11$};
\draw[thick] (1) -- (2);
\draw[thick] (3) -- (5);
\draw[thick] (5) -- (7);
\draw[thick] (4) -- (6);
\draw[thick] (6) -- (8);
\draw[thick] (9) -- (10);
\draw[thick] (10) -- (11);
\draw[thick, dashed, color=specialgold] (1) -- (3);
\draw[thick, dashed, color=specialgold] (2) -- (5);
\draw[thick, dashed, color=specialgold] (6) -- (9);
\draw[thick, dashed, color=specialgold] (8) -- (10);
\path (1) -- node[midway, color = red] {2} (5);
\path (6) -- node[midway, color = red] {3} (10);
\end{tikzpicture}
\end{center}
shows an $\FF_{p}^{\times}$-labeling $\sigma$ of the tight splice $\SSS$ from~\eqref{eq:spliceex2}, explicitly defined by
\[
\sigma\left( \binding{1}{2}{3}{5} \right) = 2\quad\text{and}\quad \sigma\left( \binding{6}{8}{9}{10} \right) = 3.
\]

\subsection{Ordering rows and columns}
\label{sec:spliceordering}

Let $\SSS$ be a splice and define
\[
\operatorname{CR}(\SSS) =  \operatorname{cols}(\SSS) \sqcup \operatorname{rows}(\SSS).
\]

\begin{prop}
\label{prop:CRsize}
Let $\SSS = \lambda \sqcup \mu$ be a splice.  Then $|\operatorname{CR}(\SSS)| = |\lambda| + |\mu| - 2|\operatorname{bind}(\SSS)|$.
\end{prop}
\begin{proof}
This follows directly from equations~\eqref{eq:colsformula} and~\eqref{eq:rowsformula}.	
\end{proof}

The main result of Section~\ref{sec:ideals} requires a uniform total order on the sets $\operatorname{CR}(\SSS)$, so as to define $\FF_{q}^{\times}$-labeled loopless binary matroids on $\operatorname{CR}(\SSS)$.  One such order is described below.

To start, draw $\SSS$ in the usual way, and then lower each successive connected component so that no row of $\SSS$ is directly to the right of another.  For example, taking $\SSS$ as in~\eqref{eq:spliceex2}, we move the second connected component down by one unit:
\begin{center}
\begin{tikzpicture}[baseline = -0.6cm]
\begin{scope}[xshift = 0cm, yshift = 0cm]
\draw[fill] (0, 0) circle (2pt) node[inner sep = 2pt] (1) {};
\draw[fill] (1, 0) circle (2pt) node[inner sep = 2pt] (2) {};
\end{scope}
\begin{scope}[xshift = 0cm, yshift = -1cm]
\draw[fill] (0, 0) circle (2pt) node[inner sep = 2pt] (3) {};
\draw[fill] (1, 0) circle (2pt) node[inner sep = 2pt] (5) {};
\draw[fill] (2, 0) circle (2pt) node[inner sep = 2pt] (7) {};
\end{scope}
\begin{scope}[xshift = 3cm, yshift = 0cm]
\draw[fill] (0, 0) circle (2pt) node[inner sep = 2pt] (4) {};
\draw[fill] (1, 0) circle (2pt) node[inner sep = 2pt] (6) {};
\draw[fill] (2, 0) circle (2pt) node[inner sep = 2pt] (8) {};
\end{scope}
\begin{scope}[xshift = 4cm, yshift = -1cm]
\draw[fill] (0, 0) circle (2pt) node[inner sep = 2pt] (9) {};
\draw[fill] (1, 0) circle (2pt) node[inner sep = 2pt] (10) {};
\draw[fill] (2, 0) circle (2pt) node[inner sep = 2pt] (11) {};
\end{scope}
\node[above] at (1) {$\scriptstyle 1$};
\node[above] at (2) {$\scriptstyle 2$};
\node[below] at (3) {$\scriptstyle 3$};
\node[above] at (4) {$\scriptstyle 4$};
\node[below] at (5) {$\scriptstyle 5$};
\node[above] at (6) {$\scriptstyle 6$};
\node[below] at (7) {$\scriptstyle 7$};
\node[above] at (8) {$\scriptstyle 8$};
\node[below] at (9) {$\scriptstyle 9$};
\node[below] at (10) {$\scriptstyle 10$};
\node[below] at (11) {$\scriptstyle 11$};
\draw[thick] (1) -- (2);
\draw[thick] (3) -- (5);
\draw[thick] (5) -- (7);
\draw[thick] (4) -- (6);
\draw[thick] (6) -- (8);
\draw[thick] (9) -- (10);
\draw[thick] (10) -- (11);
\draw[thick, dashed, color=specialgold] (1) -- (3);
\draw[thick, dashed, color=specialgold] (2) -- (5);
\draw[thick, dashed, color=specialgold] (6) -- (9);
\draw[thick, dashed, color=specialgold] (8) -- (10);
\end{tikzpicture}
\qquad$\longrightarrow$\qquad
\begin{tikzpicture}[baseline = -1.1cm]
\begin{scope}[xshift = 0cm, yshift = 0cm]
\draw[fill] (0, 0) circle (2pt) node[inner sep = 2pt] (1) {};
\draw[fill] (1, 0) circle (2pt) node[inner sep = 2pt] (2) {};
\end{scope}
\begin{scope}[xshift = 0cm, yshift = -1cm]
\draw[fill] (0, 0) circle (2pt) node[inner sep = 2pt] (3) {};
\draw[fill] (1, 0) circle (2pt) node[inner sep = 2pt] (5) {};
\draw[fill] (2, 0) circle (2pt) node[inner sep = 2pt] (7) {};
\end{scope}
\begin{scope}[xshift = 3cm, yshift = -1cm]
\draw[fill] (0, 0) circle (2pt) node[inner sep = 2pt] (4) {};
\draw[fill] (1, 0) circle (2pt) node[inner sep = 2pt] (6) {};
\draw[fill] (2, 0) circle (2pt) node[inner sep = 2pt] (8) {};
\end{scope}
\begin{scope}[xshift = 4cm, yshift = -2cm]
\draw[fill] (0, 0) circle (2pt) node[inner sep = 2pt] (9) {};
\draw[fill] (1, 0) circle (2pt) node[inner sep = 2pt] (10) {};
\draw[fill] (2, 0) circle (2pt) node[inner sep = 2pt] (11) {};
\end{scope}
\node[above] at (1) {$\scriptstyle 1$};
\node[above] at (2) {$\scriptstyle 2$};
\node[below] at (3) {$\scriptstyle 3$};
\node[above] at (4) {$\scriptstyle 4$};
\node[below] at (5) {$\scriptstyle 5$};
\node[above] at (6) {$\scriptstyle 6$};
\node[below] at (7) {$\scriptstyle 7$};
\node[above] at (8) {$\scriptstyle 8$};
\node[below] at (9) {$\scriptstyle 9$};
\node[below] at (10) {$\scriptstyle 10$};
\node[below] at (11) {$\scriptstyle 11$};
\draw[thick] (1) -- (2);
\draw[thick] (3) -- (5);
\draw[thick] (5) -- (7);
\draw[thick] (4) -- (6);
\draw[thick] (6) -- (8);
\draw[thick] (9) -- (10);
\draw[thick] (10) -- (11);
\draw[thick, dashed, color=specialgold] (1) -- (3);
\draw[thick, dashed, color=specialgold] (2) -- (5);
\draw[thick, dashed, color=specialgold] (6) -- (9);
\draw[thick, dashed, color=specialgold] (8) -- (10);
\end{tikzpicture}.
\end{center}
Label each column and row along the upper and right sides of the drawing, respectively.  Then number each label, starting from the upper left, and proceeding rightward (for columns) or down (for rows), as in
\begin{center}
\begin{tikzpicture}[baseline = -1.1cm]
\begin{scope}[xshift = 0cm, yshift = 0cm]
\draw[fill] (0, 0) circle (2pt) node[inner sep = 2pt] (1) {};
\draw[fill] (1, 0) circle (2pt) node[inner sep = 2pt] (2) {};
\end{scope}
\begin{scope}[xshift = 0cm, yshift = -1cm]
\draw[fill] (0, 0) circle (2pt) node[inner sep = 2pt] (3) {};
\draw[fill] (1, 0) circle (2pt) node[inner sep = 2pt] (5) {};
\draw[fill] (2, 0) circle (2pt) node[inner sep = 2pt] (7) {};
\end{scope}
\begin{scope}[xshift = 3cm, yshift = -1cm]
\draw[fill] (0, 0) circle (2pt) node[inner sep = 2pt] (4) {};
\draw[fill] (1, 0) circle (2pt) node[inner sep = 2pt] (6) {};
\draw[fill] (2, 0) circle (2pt) node[inner sep = 2pt] (8) {};
\end{scope}
\begin{scope}[xshift = 4cm, yshift = -2cm]
\draw[fill] (0, 0) circle (2pt) node[inner sep = 2pt] (9) {};
\draw[fill] (1, 0) circle (2pt) node[inner sep = 2pt] (10) {};
\draw[fill] (2, 0) circle (2pt) node[inner sep = 2pt] (11) {};
\end{scope}
\node[above] at (1) {$\scriptstyle 1$};
\node[above] at (2) {$\scriptstyle 2$};
\node[below] at (3) {$\scriptstyle 3$};
\node[above] at (4) {$\scriptstyle 4$};
\node[below] at (5) {$\scriptstyle 5$};
\node[above] at (6) {$\scriptstyle 6$};
\node[below] at (7) {$\scriptstyle 7$};
\node[above] at (8) {$\scriptstyle 8$};
\node[below] at (9) {$\scriptstyle 9$};
\node[below] at (10) {$\scriptstyle 10$};
\node[below] at (11) {$\scriptstyle 11$};
\draw[thick] (1) -- (2);
\draw[thick] (3) -- (5);
\draw[thick] (5) -- (7);
\draw[thick] (4) -- (6);
\draw[thick] (6) -- (8);
\draw[thick] (9) -- (10);
\draw[thick] (10) -- (11);
\draw[thick, dashed, color=specialgold] (1) -- (3);
\draw[thick, dashed, color=specialgold] (2) -- (5);
\draw[thick, dashed, color=specialgold] (6) -- (9);
\draw[thick, dashed, color=specialgold] (8) -- (10);
\begin{scope}[xshift = 7cm, yshift = 0.6cm]
\node at (0, -0.9) {$\RRR_{1}$};
\node at (0, -2.1) {$\RRR_{2}$};
\end{scope}
\begin{scope}[xshift = -0.5cm, yshift = 0cm]
\node at (1, 1) {$\CCC_{1}$};
\node at (2, 1) {$\CCC_{2}$};
\node at (4, 1) {$\CCC_{3}$};
\node at (5, 1) {$\CCC_{4}$};
\node at (6, 1) {$\CCC_{5}$};
\end{scope}
\end{tikzpicture}
\end{center}
with 
\[
\CCC_{1} = \left\{ \edge{1}{2}, \edge{3}{5}\right\},\; \CCC_{2} = \left\{ \edge{5}{7} \right\},\;\CCC_{3} = \left\{ \edge{4}{6} \right\},\; \CCC_{4} = \left\{ \edge{6}{8}, \edge{9}{10} \right\},
\]
\[
\CCC_{5} =  \left\{ \edge{10}{11} \right\},\;\RRR_{1} = \left\{ \dedge{1}{3}\;, \dedge{2}{5} \right\},\;\text{and}\;\RRR_{2} = \left\{ \dedge{6}{9}\;, \dedge{8}{10} \right\}.
\]
Finally, enumerate the set $\operatorname{CR}(\SSS)$ by first listing columns according to number, and then rows in the same way.  This enumeration defines the final ordering, as in:
\[
\CCC_{1} < \CCC_{2} < \CCC_{3} < \CCC_{4} < \CCC_{5} < \RRR_{1} < \RRR_{2}.
\]

\begin{rem}
Aside from ease of description, I do not know of any benefit of this order over any other, at least in the scope of this paper.  However, the analogous choice of order for shortened polyominoes is fairly significant (see~\cite[Section 5]{SzuEnEtAl13}), so it would be interesting to know of a property of ideals or normal subgroups which prefers one order over another.
\end{rem}

\section{Lie ideals}
\label{sec:ideals}

Fix a prime power $q$.  This section describes a bijective indexing for the ideals of the Lie algebra $\ut_{n} = \ut_{n}(\FF_{q})$.  Recall the definitions of $\mathscr{T}_{n}(q)$ from Section~\ref{sec:labelingsplices}, $\operatorname{CR}(\SSS)$ from Section~\ref{sec:spliceordering}, and $\overset{\to}{\mathscr{G}}_{\operatorname{CR}(\SSS)}(q)$ from Section~\ref{sec:matroidgraphs}.

\begin{thm}\label{thm:ideals}
There is a bijection
\[
 \{\text{Lie algebra ideals of $\ut_{n}$}\}  \longleftrightarrow \left\{(\SSS, \sigma, \MMM, \tau) \;\middle|\; (\SSS, \sigma) \in \mathscr{T}_{n}(q), (\MMM, \tau) \in \overset{\to}{\mathscr{G}}_{\operatorname{CR}(\SSS)}(q)  \right\}
\]
\end{thm}

With Theorem~\ref{introthm:A}, this implies Theorem~\ref{introthm:B}.  With the formula for $|\overset{\to}{\mathscr{G}}_{\operatorname{CR}(\SSS)}(q)|$ from Proposition~\ref{prop:matroidcount}, the next result also follows.

\begin{cor}\label{cor:numideals}
The number of Lie algebra ideals in $\ut_{n}$ is
\[
\sum_{\SSS \in \mathscr{T}_{n}} (q-1)^{|\operatorname{bind}(\SSS)|}
\sum_{j = 0}^{|\operatorname{CR}(\SSS)|} (q-1)^{|\operatorname{CR}(\SSS)|-j} \; \qstirling{|\operatorname{CR}(\SSS)|}{j}_{q}.
\]
\end{cor}

The aim of this section is to prove Theorem~\ref{thm:ideals}; an outline of the proof structure follows.
\begin{enumerate}
\item Each labeled tight splice $(\SSS, \sigma) \in \mathscr{T}_{n}(q)$ selects a family of ideals, denoted $\operatorname{fam}(\SSS, \sigma)$.  Proposition~\ref{prop:splicecondition} states that
\[
\{\text{Lie algebra ideals of $\ut_{n}$}\} = \bigsqcup_{(\SSS, \sigma) \in \mathscr{T}_{n}(q)} \hspace{-1em}\operatorname{fam}(\SSS, \sigma).
\]
For each $(\SSS, \sigma)$, there are ideals $\mathfrak{D}_{\SSS, \sigma}$ and $\mathfrak{Z}_{\SSS, \sigma}$ such that $\mathfrak{D}_{\SSS, \sigma} \subseteq \mathfrak{n} \subseteq \mathfrak{Z}_{\SSS, \sigma}$ for each $\mathfrak{n} \in \operatorname{fam}(\SSS, \sigma)$.  Every $\mathfrak{n} \in \operatorname{fam}(\SSS, \sigma)$ is then determined by its image in $\mathfrak{Z}_{\SSS, \sigma}/\mathfrak{D}_{\SSS, \sigma}$.

\item For each $(\SSS, \sigma) \in \mathscr{T}_{n}(q)$, Proposition~\ref{prop:matroidbijection} gives a explicit bijection
\[
 \overset{\to}{\mathscr{G}}_{\operatorname{CR}(\SSS)}(q) \longleftrightarrow \operatorname{fam}(\SSS, \sigma).
\]
Each $\FF_{q}^{\times}$-labeled loopless binary matroid $(\MMM, \tau)$ on $\operatorname{CR}(\SSS)$ uniquely determines a subspace of $\mathfrak{Z}_{\SSS, \sigma}/\mathfrak{D}_{\SSS, \sigma}$ which lifts to an ideal in $\operatorname{fam}(\SSS, \sigma)$.

\end{enumerate}

\begin{rem}
The notation used for the ideals $\mathfrak{Z}_{\SSS, \sigma}$ and $\mathfrak{D}_{\SSS, \sigma}$ is meant to evoke the terms center and derivation: Proposition~\ref{prop:CandZ} and Lemma~\ref{lem:Zcommutator} respectively state that for $\mathfrak{n} \in \operatorname{fam}(\SSS, \sigma)$,
\[
Z(\ut_{n}/\mathfrak{n}) = \mathfrak{Z}_{\SSS, \sigma}/\mathfrak{n}, \quad\text{and}\quad [\ut_{n}, \mathfrak{Z}_{\SSS, \sigma}] = \mathfrak{D}_{\SSS, \sigma}.
\]
\end{rem}

\subsection{Splices and families of ideals}
\label{sec:splicecondition}

For $\lambda \in \mathrm{NNSP}_{n}$ and $(\SSS, \sigma) \in \mathscr{T}_{\lambda}(q)$, define a subspace of $\ut_{n}$
\[
\mathfrak{Z}_{\SSS, \sigma} = \left\{a \in \ut_{\lambda} \;\middle|\;  a_{i, j} = \sigma\left(\binding{i}{j}{j+1}{l}\right)a_{j+1, l} \;\text{ for each}\; \binding{i}{j}{j+1}{l} \in \operatorname{bind}(\SSS)\right\}.
\]
In $\mathfrak{Z}_{\SSS, \sigma}$, matrix entries corresponding to elements of the same column of $\SSS$ have a fixed ratio.  Section~\ref{sec:CandZ} gives a more in-depth description, including a basis and several examples.

Now suppose that $\SSS = \lambda \sqcup \nu$, and define the \emph{$(\SSS, \sigma)$-family} of ideals
\[
\operatorname{fam}(\SSS, \sigma) = \left\{\text{ideals $\mathfrak{n}$ of $\ut_{n}$} \;\middle|\;  \begin{array}{c}
\text{$\operatorname{supp}(\mathfrak{n}) = \up{\lambda}$,} \\
\text{$\{{\textstyle \edge{i}{j}} \in \upnum{\lambda}{1} \;|\; \FF_{q}e_{i, j} \not\subseteq \mathfrak{n}\} = \nu$,}\\
\text{and $\mathfrak{n} \subseteq \mathfrak{Z}_{\SSS, \sigma}$}
\end{array}\right\}.
\]

\begin{prop}\label{prop:splicecondition}
Each ideal of $\ut_{n}$ is contained in a unique family:
\[
\{\text{ideals of $\ut_{n}$\}} = \bigsqcup_{(\SSS, \sigma) \in \mathscr{T}_{n}(q)} \operatorname{fam}(\SSS, \sigma).
\]
\end{prop}

Proposition~\ref{prop:splicecondition} constitutes the first half of Theorem~\ref{thm:ideals}.  The proof of this proposition will follow a sequence of three lemmas.

\begin{lem}\label{lem:upnum2}
Let $\mathfrak{n}$ be an ideal of $\ut_{n}$, and take $\lambda \in \mathrm{NNSP}_{n}$ with $\up{\lambda} = \operatorname{supp}(\mathfrak{n})$. Then
\[
[\ut_{n}, [\ut_{n}, \mathfrak{n}]] = \ut_{\upnum{\lambda}{2}}.
\]
\end{lem}

\begin{proof}
By assumption $\mathfrak{n} \subseteq \ut_{\lambda}$, so
\[
[\ut_{n}, [\ut_{n}, \mathfrak{n}]]  \subseteq [\ut_{n}, [\ut_{n}, \ut_{\lambda}]] = \ut_{\upnum{\lambda}{2}},
\]
establishing one direction of containment.  The other will follow from the construction of a basis for $\ut_{\upnum{\lambda}{2}}$ within $[\ut_{n}, [\ut_{n}, \mathfrak{n}]]$.  Given any $\edge{i}{l} \in \upnum{\lambda}{2}$, it is possible to choose
\begin{itemize}
\item $\edge{j}{k} \in \lambda$ with $\edge{j}{k} \preceq \edge{i}{l}$ and $ \Ht(\edge{j}{k}) + 2 \le \Ht(\edge{i}{l})$, and 

\item $a \in \ut_{n}$ with $a_{j, k} \neq 0$.

\end{itemize}
Using these choices, define a new element for each $\edge{i}{l} \in \upnum{\lambda}{2}$, 
\[
b(i, l) = \begin{cases}
[e_{k+1, l}, [e_{k, k+1}, a]] =  \sum_{r \le j} a_{r, k}e_{r, l} & \text{if $i = j < k < l$,}\\
[e_{i, j}, [e_{k, l}, a]] =  (-a_{j, k}) e_{i, l} & \text{if $i < j < k < l$,}\\
[e_{i, j-1}, [e_{j-1, j}, a]] = \sum_{s \ge k}  a_{j, s}e_{i, s} & \text{if $i < j < k = l$.}
\end{cases}
\]
The $i, l$-entry of $b(i, l)$ is necessarily nonzero and $\operatorname{supp}(\{b(i, l)\}) \subseteq \up{\{\edge{i}{l}\}}$, so the transition from $\{b(i, l) \;|\; \edge{i}{l} \in \upnum{\lambda}{2}\}$ to $\{e_{i, l} \;|\; \edge{i}{l} \in \upnum{\lambda}{2}\}$ is nonsingular and triangular with respect to $\preceq$.  Thus $\{b(i, l) \;|\; \edge{i}{l} \in \upnum{\lambda}{2}\}$ is a basis for $\ut_{\upnum{\lambda}{2}}$.
\end{proof}

\begin{lem}\label{lem:commutator}
Fix $\lambda \in \mathrm{NNSP}_{n}$, and take $a \in \ut_{\lambda}$.  For any $\edge{j}{j+1}\in[[n]]$,  
\[
[\FF_{q}e_{j, j+1}, a] + \ut_{\upnum{\lambda}{2}} = \begin{cases}
\FF_{q}(a_{i, j}e_{i, j+1} - a_{j+1, l}e_{j, l}) + \ut_{\upnum{\lambda}{2}} & \begin{array}{l}  \text{if $\edge{i}{j}, \edge{j+1}{l} \in \lambda$,} \end{array}\\ 
\FF_{q}(a_{i, j}e_{i, j+1}) + \ut_{\upnum{\lambda}{2}} & \begin{array}{l} \text{if $\edge{i}{j} \in \lambda$ and $\edge{j+1}{l} \notin \lambda$} \\  \quad \text{for all $l \in [n]$,} \end{array}\\
\FF_{q}(a_{j+1, l}e_{j, l}) + \ut_{\upnum{\lambda}{2}} & \begin{array}{l} \text{if $\edge{j+1}{l} \in \lambda$ and $\edge{i}{j} \notin \lambda$} \\ \quad \text{for all $i \in [n]$,} \end{array}\\
 \ut_{\upnum{\lambda}{2}} & \begin{array}{l} \text{otherwise.} \end{array}
\end{cases}
\]
\end{lem}

Lemma~\ref{lem:commutator} should be seen more as a computational shortcut than a precise description of the set $[\FF_{q}e_{j, j+1}, a] + \ut_{\upnum{\lambda}{2}}$: in the event that one or both of $a_{j+1, l}$ or $a_{i, j}$ is zero, several of the cases above are equal.

\begin{proof}
To begin, note that $[\ut_{n}, \ut_{\lambda}] = \ut_{\upnum{\lambda}{1}}$, so $[\ut_{n}, a] \subseteq \ut_{\upnum{\lambda}{1}}$  for any $a \in \ut_{\lambda}$.  Every element of $\ut_{\upnum{\lambda}{1}}/\ut_{\upnum{\lambda}{2}}$ has a canonical representative with all nonzero entries occurring in positions indexed by elements of $\upnum{\lambda}{1} \setminus \upnum{\lambda}{2}$, and it is these representatives which appear in the statement of the lemma.

For $\edge{j}{j+1} \in [[n]]$ and any $a \in \ut_{\lambda}$, 
\[
[\FF_{q}e_{j, j+1}, a] = \FF_{q}\left( \sum_{l > j+1} a_{j+1, l}e_{j, l} - \sum_{i < j} a_{i, j}e_{i, j+1} \right) = \FF_{q}\left(\sum_{i < j} a_{i, j}e_{i, j+1} -  \sum_{l > j+1} a_{j+1, l}e_{j, l} \right).
\]
Passing to the quotient $\ut_{\lambda}/\ut_{\upnum{\lambda}{2}}$, the canonical representative of $[\FF_{q}e_{j, j+1}, a] + \ut_{\upnum{\lambda}{2}}$ can have nonzero $i, j+1$-entry only if $\edge{i}{j} \in \up{\lambda}$ and $\edge{i}{j+1} \notin \upnum{\lambda}{2}$; together these conditions imply $\edge{i}{j} \in \lambda$.  Similarly, the representative has nonzero $j, l$-entry only if $\edge{j+1}{l} \in \lambda$.
\end{proof}

The lemma is a slightly stronger statement than Proposition~\ref{prop:splicecondition} which will also apply to normal subgroups, as discussed in Section~\ref{sec:normalsubgroups}.

\begin{lem}\label{lem:splicecondition}
For $\lambda \in \mathrm{NNSP}_{n}$, let $\mathfrak{n}$ be an arbitrary subset of $\ut_{\lambda}$, and suppose that $\mathfrak{m}$ is an additively closed subset of $\ut_{n}$ which satisfies
\[
\{[a, b] \;|\; a \in \ut_{n}, b \in \mathfrak{n}\} + \ut_{\upnum{\lambda}{2}} \subseteq \mathfrak{m} \subseteq \ut_{\upnum{\lambda}{1}}.
\]
Let $\nu = \{{\textstyle \edge{r}{s}} \in \upnum{\lambda}{1} \;|\; \FF_{q}e_{r, s} \not\subseteq \mathfrak{m}\}$.  Then $\SSS = \lambda \sqcup \nu$ is a tight splice of $\lambda$, and there is a unique labeling $\sigma: \operatorname{bind}(\SSS) \to \FF_{q}^{\times}$ for which $\mathfrak{n} \subseteq \mathfrak{Z}_{\SSS, \sigma}$.
\end{lem}

\begin{proof}[Proof of Lemma~\ref{lem:splicecondition}]

The first step of the proof is to show that $\SSS$ is a tight splice. From the assumption that $\ut_{\upnum{\lambda}{2}} \subseteq \mathfrak{m}$ it follows that $\nu \subseteq \upnum{\lambda}{1} \setminus\hspace{-0.25cm} \upnum{\lambda}{2}$, so $\SSS$ satisfies conditions~\ref{S1} and~\eqref{T}:
\[
\text{if $\dedge{r}{s} \in \nu$} \quad\text{then}\quad \text{$\edge{r+1}{s}$ or $\edge{r}{s-1} \in \lambda$}.
\]

What remains in this step is to show that $\SSS$ satisfies~\ref{S2}.  
For each $\edge{i}{j} \in \lambda$ we can choose an $a \in \mathfrak{n}$ with $a_{i, j} \neq 0$.  Using Lemma~\ref{lem:commutator} to assist with computation, $\mathfrak{m}$ contains
\[
[\FF_{q}e_{j, j+1}, a] + \ut_{\upnum{\lambda}{2}} = \begin{cases}
\FF_{q}(a_{i, j}e_{i, j+1} - a_{j+1, l}e_{j, l}) + \ut_{\upnum{\lambda}{2}} & \text{if $\edge{j+1}{l} \in \lambda$,} \\ 
\FF_{q}e_{i, j+1} + \ut_{\upnum{\lambda}{2}} & \text{if  $\edge{j+1}{l} \notin \lambda$ for all $l \in [n]$.}\end{cases}
\]

If $\edge{i}{j+1} \in \nu$ then in order to avoid a contradiction there must be some $l \in [n]$ with $\edge{j+1}{l} \in \lambda$, $a_{j, l} \neq 0$, and $\FF_{q}e_{j, l} \notin \mathfrak{m}$.  As a consequence, $\edge{j}{l} \in \nu$, establishing
\[
\edge{i}{j} \in \lambda,\;\dedge{i}{j+1} \in \nu \quad\text{only if}\quad \edge{j+1}{l} \in \lambda,\;\dedge{j}{l} \in \nu\; \text{for some $l \in [n]$}.
\]
The ``if'' direction of~\ref{S2} follows similarly, establishing that $\SSS$ is a tight splice of $\lambda$.

The second and final step of the proof is to define a labeling $\sigma$ of $\SSS$ for which $\mathfrak{n} \subseteq \mathfrak{Z}_{\SSS, \sigma}$.  Fix a binding
\[
\binding{i}{j}{j+1}{l} \in \operatorname{bind}(\SSS),\quad \text{with}\;\dedge{i}{j+1}\;, \dedge{j}{l} \in \nu \;\text{and}\; \edge{i}{j}, \edge{j+1}{l} \in \lambda.
\]
Using Lemma~\ref{lem:commutator} again, for each $a \in \mathfrak{n}$
\[
[\FF_{q}e_{j, j+1}, a] + \ut_{\upnum{\lambda}{2}} = \FF_{q}(a_{i, j}e_{i, j+1} - a_{j+1, l}e_{j, l}) + \ut_{\upnum{\lambda}{2}}  \subseteq \mathfrak{m}.
\]
Since $\mathfrak{m}$ is closed under addition and $\FF_{q}e_{i, j+1}, \FF_{q}e_{j, l} \not\subseteq \mathfrak{m}$, there must be a single value $\alpha_{i, j, j+1, l} \in \FF_{q}^{\times}$ for which
\[
a_{i, j}e_{i, j+1} - a_{j+1, l}e_{j, l} \in \FF_{q}(e_{i, j+1} - \alpha_{i, j, j+1, k}e_{j, l}) 
\]
for every $a \in \mathfrak{n}$.  Such a value exists for each binding of $\SSS$, giving a labeling $\sigma$ of $\operatorname{bind}(\SSS)$ with
\[
\sigma\left(\binding{i}{j}{j+1}{l} \right) = \alpha_{i, j, j+1, l}. \qedhere
\]
\end{proof}

\begin{proof}[Proof of Proposition~\ref{prop:splicecondition}]
Suppose that $\mathfrak{n}$ is an ideal with support $\up{\lambda}$ for $\lambda \in \mathrm{NNSP}_{n}$, and let $\mathfrak{m} = \mathfrak{n} \cap \ut_{n}$, so that $\mathfrak{m}$ is an ideal, and
\[
\{{\textstyle \edge{r}{s}} \in \upnum{\lambda}{1} \;|\; \FF_{q}e_{r, s} \not\subseteq \mathfrak{m}\} = \{{\textstyle \edge{r}{s}} \in \upnum{\lambda}{1} \;|\; \FF_{q}e_{r, s} \not\subseteq \mathfrak{n}\}.
\]
By Lemma~\ref{lem:upnum2},  $\ut_{\upnum{\lambda}{2}} \subseteq [\ut_{n}, \mathfrak{n}]$. As $\operatorname{supp}(\mathfrak{n}) = \up{\lambda}$, $[\ut_{n}, \mathfrak{n}] \subseteq \ut_{\upnum{\lambda}{1}} \cap \mathfrak{n}$.  Therefore $\mathfrak{n}$ and $\mathfrak{m}$ fulfill the assumptions of Lemma~\ref{lem:splicecondition}.  The result is a unique element $(\SSS, \sigma) \in \mathscr{T}_{n}(q)$ for which $\mathfrak{n} \in \operatorname{fam}(\SSS, \sigma)$, proving the proposition.
\end{proof}

\subsubsection{Bounding ideals}
\label{sec:CandZ}

Let $(\SSS, \sigma) \in \mathscr{T}_{n}(q)$ with $\SSS = \lambda \sqcup \nu$.  By definition each ideal $\mathfrak{n} \in \operatorname{fam}(\SSS, \sigma)$ is contained in
\[
\mathfrak{Z}_{\SSS, \sigma} = \left\{a \in \ut_{\lambda} \;\middle|\;  a_{i, j} = \sigma\left(\binding{i}{j}{j+1}{l}\right)a_{j+1, l} \;\text{ for each}\; \binding{i}{j}{j+1}{l} \in \operatorname{bind}(\SSS)\right\}.
\]
There is an analogous lower bound for each family.  Let
\[
\mathfrak{D}_{\SSS, \sigma} = \ut_{\up{\lambda} \setminus \SSS} \oplus \hspace{-0.5cm}\bigoplus_{\subbinding{i}{j}{j+1}{l}  \in \operatorname{bind}(\SSS)}\hspace{-0.4cm} \FF_{q} \left(e_{i, j+1} - \sigma\left(\binding{i}{j}{j+1}{l}\right)e_{j, l}  \right).
\]
A piece of combinatorial intuition for these definitions is as follows: in $\mathfrak{Z}_{\SSS, \sigma}$, we consider entries from $\lambda$ of $\SSS$ as related by $\sim_{\text{cols}}$, while in $\mathfrak{D}_{\SSS, \sigma}$ we consider entries from $\nu$ as related by $\sim_{\text{rows}}$.

Both $\mathfrak{Z}_{\SSS, \sigma}$ and $\mathfrak{D}_{\SSS, \sigma}$ are ideals of $\ut_{n}$, as
\[
\ut_{\upnum{\lambda}{2}} \subseteq \mathfrak{D}_{\SSS, \sigma} \subseteq \ut_{\upnum{\lambda}{1}} \subseteq \mathfrak{Z}_{\SSS, \sigma} \subseteq \ut_{\lambda}.
\]
This also shows that $\mathfrak{Z}_{\SSS, \sigma} \in \operatorname{fam}(\lambda, f_{\emptyset})$ and $\mathfrak{D}_{\SSS, \sigma} \in \operatorname{fam}(\min(\upnum{\lambda}{1}), f_{\emptyset})$, where $f_{\emptyset}$ is the unique labelling of a trivial splice, which has no bindings.  It is therefore typical that neither $\mathfrak{Z}_{\SSS, \sigma}$ nor $\mathfrak{D}_{\SSS, \sigma}$ belong to $\operatorname{fam}(\SSS, \sigma)$.

\begin{ex}\label{ex:CandZ}
For $n = 8$ and $q = 5$, take $\SSS = \lambda \sqcup \nu$ and $\sigma: \operatorname{bind}(\SSS) \to \FF_{5}^{\times}$ as given by
\begin{center}
\begin{tikzpicture}[scale = 1, baseline = -0.1cm]
\begin{scope}[xshift = 0cm, yshift = 1cm]
\draw[fill = black] (0, 0) circle (2pt) node[inner sep = 2pt] (1) {};
\draw[fill = black] (1, 0) circle (2pt) node[inner sep = 2pt] (2) {};
\end{scope}
\begin{scope}[xshift = 0cm, yshift = 0cm]
\draw[fill = black] (0, 0) circle (2pt) node[inner sep = 2pt] (3) {};
\draw[fill = black] (1, 0) circle (2pt) node[inner sep = 2pt] (4) {};
\draw[fill = black] (2, 0) circle (2pt) node[inner sep = 2pt] (6) {};
\end{scope}
\begin{scope}[xshift = 0cm, yshift = -1cm]
\draw[fill = black] (0, 0) circle (2pt) node[inner sep = 2pt] (5) {};
\draw[fill = black] (1, 0) circle (2pt) node[inner sep = 2pt] (7) {};
\draw[fill = black] (2, 0) circle (2pt) node[inner sep = 2pt] (8) {};
\end{scope}
\draw[thick, color = black] (1) -- (2);
\draw[thick, color = black] (3) -- (4);
\draw[thick, color = black] (4) -- (6);
\draw[thick, color = black] (5) -- (7);
\draw[thick, color = black] (7) -- (8);
\draw[color = specialgold, dashed, thick] (1)-- (3);
\draw[color = specialgold, dashed, thick] (2) -- (4);
\draw[color = specialgold, dashed, thick] (3) -- (5);
\draw[color = specialgold, dashed, thick] (4) -- (7);
\draw[color = specialgold, dashed, thick] (6) -- (8);
\path (1) -- node[pos = 0.5, color = red] {$3$} (4);
\path (3) -- node[pos = 0.5, color = red] {$1$} (7);
\path (4) -- node[pos = 0.5, color = red] {$2$} (8);
\node[above] at (1) {$\scriptstyle 1$};
\node[above] at (2) {$\scriptstyle 2$};
\node[above left] at (3) {$\scriptstyle 3$};
\node[above right] at (4) {$\scriptstyle 4$};
\node[below] at (5) {$\scriptstyle 5$};
\node[above right] at (6) {$\scriptstyle 6$};
\node[below] at (7) {$\scriptstyle 7$};
\node[below] at (8){$\scriptstyle 8$};
\end{tikzpicture}.
\end{center}
Therefore
\begin{align*}
\mathfrak{Z}_{\SSS, \sigma} &= \ut_{\upnum{\lambda}{1}} \oplus \FF_{5}(e_{1, 2} + 3e_{3, 4} + 3e_{5, 7}) \oplus \FF_{5}(e_{4, 6} + 2e_{7, 8}) \\
&= 
\begin{tikzpicture}[scale = 0.5, baseline = -4*0.5cm]
\draw[thick] (0.3, 0.1) -- (-0, 0.1) -- (-0, -8.1) -- (0.3, -8.1);
\foreach \x in {0, 1, 2, 3, 4, 5, 6, 7}{\node at ({0.5, -\x -0.5}) {$0$};}
\node at (1.5, -0.5) {$\alpha$};
\foreach \x in {1, 2, 3, 4, 5, 6, 7}{\node at ({1.5, -\x -0.5}) {$0$};}
\node at (2.5, -0.5) {$\ast$};
\foreach \x in {1, 2, 3, 4, 5, 6, 7}{\node at ({2.5, -\x -0.5}) {$0$};}
\begin{scope}[xshift = 0.1cm]
\foreach \x in {0, 1}{\node at ({3.5, -\x -0.5}) {$\ast$};}
\node at (3.5, -2.5) {$3\alpha$};
\foreach \x in {3, 4, 5, 6, 7}{\node at ({3.5, -\x -0.5}) {$0$};}
\foreach \x in {0, 1, 2}{\node at ({4.5, -\x -0.5}) {$\ast$};}
\foreach \x in {3, 4, 5, 6, 7}{\node at ({4.5, -\x -0.5}) {$0$};}
\foreach \x in {0, 1, 2}{\node at ({5.5, -\x -0.5}) {$\ast$};}
\node at (5.5, -3.5) {$\beta$};
\foreach \x in {4, 5, 6, 7}{\node at ({5.5, -\x -0.5}) {$0$};}
\begin{scope}[xshift = 0.2cm]
\foreach \x in {0, 1, 2, 3}{\node at ({6.5, -\x -0.5}) {$\ast$};}
\node at (6.5, -4.5) {$3\alpha$};
\foreach \x in {5, 6, 7}{\node at ({6.5, -\x -0.5}) {$0$};}
\begin{scope}[xshift = 0.2cm]
\foreach \x in {0, 1, 2, 3, 4, 5}{\node at ({7.5, -\x -0.5}) {$\ast$};}
\node at (7.5, -6.5) {$2\beta$};
\foreach \x in {7}{\node at ({7.5, -\x -0.5}) {$0$};}
\draw[thick] (7.8, 0.1) -- (8.1, 0.1) -- (8.1, -8.1) -- (7.8, -8.1);
\end{scope}
\end{scope}
\end{scope}
\draw[thick] (0.2, 0) -- (1, 0) -- (1, -1) -- (3, -1) -- (3, -3) -- (5.2, -3) -- (5.2, -4) -- (6.2, -4) -- (6.2, -5) -- (7.3, -5) -- (7.3, -7) -- (8.5, -7) -- (8.5, - 7.95);
\end{tikzpicture}\;\;\text{over all values of $\alpha, \beta \in \FF_{5}$,}
\end{align*}
and $\up{\lambda} \setminus \SSS = \{\edge{1}{4}$, $\edge{1}{5}$, $\edge{1}{6}$, $\edge{1}{7}$, $\edge{1}{8}$, $\edge{2}{5}$, $\edge{2}{6}$, $\edge{2}{7}$, $\edge{2}{8}$, $\edge{3}{6}$, $\edge{3}{7}$, $\edge{3}{8}$, $\edge{4}{8}$, $\edge{5}{8} \}$, so
\begin{align*}
\mathfrak{D}_{\SSS, \sigma} &= \ut_{\up{\lambda} \setminus \SSS} \oplus \FF_{5}(e_{1, 3} - 3e_{2, 4}) \oplus \FF_{5}(e_{3, 5} - e_{4, 7}) \oplus \FF_{5}(e_{4, 7} - 2e_{6, 8}) \\[0.5em]
&= 
\begin{tikzpicture}[scale = 0.5, baseline = -4*0.5cm]
\draw[thick] (0.3, 0.1) -- (-0, 0.1) -- (-0, -8.1) -- (0.3, -8.1);
\foreach \x in {0, 1, 2, 3, 4, 5, 6, 7}{\node at ({0.5, -\x -0.5}) {$0$};}
\foreach \x in {0, 1, 2, 3, 4, 5, 6, 7}{\node at ({1.5, -\x -0.5}) {$0$};}
\node at (2.5, -0.5) {$\gamma$};
\foreach \x in {1, 2, 3, 4, 5, 6, 7}{\node at ({2.5, -\x -0.5}) {$0$};}
\begin{scope}[xshift = 0.3cm]
\foreach \x in {0}{\node at ({3.5, -\x -0.5}) {$\ast$};}
\node at (3.5, -1.5) {$-3\gamma$};
\foreach \x in {2, 3, 4, 5, 6, 7}{\node at ({3.5, -\x -0.5}) {$0$};}
\begin{scope}[xshift = 0.3cm]
\foreach \x in {0, 1}{\node at ({4.5, -\x -0.5}) {$\ast$};}
\node at (4.5, -2.5) {$\delta$};
\foreach \x in {3, 4, 5, 6, 7}{\node at ({4.5, -\x -0.5}) {$0$};}
\foreach \x in {0, 1, 2}{\node at ({5.5, -\x -0.5}) {$\ast$};}
\foreach \x in {3, 4, 5, 6, 7}{\node at ({5.5, -\x -0.5}) {$0$};}
\begin{scope}[xshift = 0.6cm]
\foreach \x in {0, 1, 2}{\node at ({6.5, -\x -0.5}) {$\ast$};}
\node at (6.5, -3.5) {$\epsilon - \delta$};
\foreach \x in {4, 5, 6, 7}{\node at ({6.5, -\x -0.5}) {$0$};}
\begin{scope}[xshift = 0.9cm]
\foreach \x in {0, 1, 2, 3, 4}{\node at ({7.5, -\x -0.5}) {$\ast$};}
\node at (7.5, -5.5) {$-2\epsilon$};
\foreach \x in {6, 7}{\node at ({7.5, -\x -0.5}) {$0$};}
\begin{scope}[xshift = 0.3cm]
\draw[thick] (7.8, 0.1) -- (8.1, 0.1) -- (8.1, -8.1) -- (7.8, -8.1);
\end{scope}
\end{scope}
\end{scope}
\end{scope}
\end{scope}
\draw[thick, specialgold] (0.2, 0) -- (2, 0) -- (2, -1) -- (3, -1) -- (3, -2) -- (4.6, -2) -- (4.6, -3) -- (6.6, -3) -- (6.6, -4) -- (8.8, -4) -- (8.8, -6) -- (10.4, -6) -- (10.4, -7.95);
\draw[thick] (0.2, 0) -- (1, 0) -- (1, -1) -- (3, -1) -- (3, -3) -- (5.6, -3) -- (5.6, -4) -- (6.6, -4) -- (6.6, -5) -- (8.8, -5) -- (8.8, -7) -- (10.4, -7) -- (10.4, -7.95);
\end{tikzpicture}
\;\;\text{over all values of $\gamma, \delta, \epsilon \in \FF_{5}$.}
\end{align*}
Following the convention of Example~\ref{ex:patternthings}, I have drawn paths in each matrix above to emphasize the support of each ideal, with two paths in $\mathscr{D}_{\SSS, \sigma}$ to emphasize the difference in support from $\mathfrak{Z}_{\SSS, \sigma}$, reflecting the general fact that $\operatorname{supp}(\mathfrak{Z}_{\SSS, \sigma}) = \up{\lambda}$ and $\operatorname{supp}(\mathfrak{D}_{\SSS, \sigma}) = \upnum{\lambda}{1}$.
\end{ex}

\begin{lem}\label{lem:Zcommutator}
For any $(\SSS, \sigma) \in \mathscr{T}_{n}(q)$, $[\ut_{n}, \mathfrak{Z}_{\SSS, \sigma}] = \mathfrak{D}_{\SSS, \sigma}$.
\end{lem}

\begin{proof}
By Lemma~\ref{lem:upnum2}, $\ut_{\upnum{\lambda}{2}} \subseteq [\ut_{n}, \mathfrak{Z}_{\SSS, \sigma}]$.  The span of the commutators of $\mathfrak{Z}_{\SSS, \sigma}$ which are not contained in $\ut_{\upnum{\lambda}{2}}$ can then be computed using Lemma~\ref{lem:commutator}, as in the proof of Lemma~\ref{lem:splicecondition}:
\[
[\ut_{n}, \mathfrak{Z}_{\SSS, \sigma}] = \ut_{\up{\lambda} \setminus \SSS} \oplus \hspace{-0.5cm}  \bigoplus_{\subbinding{i}{j}{j+1}{l}  \in \operatorname{bind}(\SSS)}\hspace{-0.4cm} \FF_{q}\left(e_{i, j+1} - \sigma\left(\binding{i}{j}{j+1}{l} \right)e_{j, l} \right) = \mathfrak{D}_{\SSS, \sigma}.
\qedhere
\]
\end{proof}

\begin{prop}\label{prop:CandZ}
For $(\SSS, \sigma) \in \mathscr{T}_{n}(q)$ and $\mathfrak{n} \in \operatorname{fam}(\SSS, \sigma)$, $Z(\ut_{n}/\mathfrak{n}) = \mathfrak{Z}_{\SSS, \sigma}/\mathfrak{n}$, and as a consequence $\mathfrak{D}_{\SSS, \sigma} \subseteq \mathfrak{n} \subseteq \mathfrak{Z}_{\SSS, \sigma}$.
\end{prop}

The converse is almost never true.  If $\SSS \neq \emptyset$, there are ideals between $\mathfrak{D}_{\SSS, \sigma}$ and $\mathfrak{Z}_{\SSS, \sigma}$ which do not belong to $\operatorname{fam}(\SSS, \sigma)$; one example is $\ut_{\upnum{\lambda}{1}}$, where $\lambda$ is the underlying partition of $\SSS$.

\begin{proof}
Write $Z = \{a \in \ut_{\lambda} \;|\; [\ut_{n}, a] \subseteq \mathfrak{n}\}$.  Assuming that $Z = \mathfrak{Z}_{\SSS, \sigma}$, it follows that $\mathfrak{n} \subseteq \mathfrak{Z}_{\SSS, \sigma}$, from which Lemma~\ref{lem:Zcommutator} implies $\mathfrak{D}_{\SSS, \sigma} \subseteq \mathfrak{n}$.  The goal, therefore, is to show that $Z = \mathfrak{Z}_{\SSS, \sigma}$.

By Lemma~\ref{lem:upnum2}, $\ut_{\upnum{\lambda}{2}} \subseteq \mathfrak{n}$, so $\ut_{\upnum{\lambda}{1}} \subseteq Z$.  Thus, $a \in Z$ if and only if $[\ut_{n}, a] + \ut_{\upnum{\lambda}{2}} \subseteq \mathfrak{n}$.  Using Lemma~\ref{lem:commutator} to compute $[\ut_{n}, a] + \ut_{\upnum{\lambda}{2}}$,
\[
a \in Z \quad\text{if and only if} \quad \text{$\FF_{q}(a_{i, j}e_{i, j+1} - a_{j+1, l}e_{j, l}) \subseteq \mathfrak{n}$ for each $\binding{i}{j}{j+1}{l} \in \operatorname{bind}(\SSS)$}.
\]
Since $\mathfrak{n} \in \operatorname{fam}(\SSS, \sigma)$, it must be the case that $\FF_{q}e_{i, j+1}, \FF_{q}e_{j, l} \not\subseteq \mathfrak{n}$ and
\[
e_{i, j+1} - \sigma\left(\binding{i}{j}{j+1}{l} \right) e_{j, l} \in \mathfrak{n}\qquad\text{for each}\quad \binding{i}{j}{j+1}{l} \in \operatorname{bind}(\SSS),
\]
so $a \in Z$ if and only if $a \in \mathfrak{Z}_{\SSS, \sigma}$.
\end{proof}

Now consider the quotient $\mathfrak{Z}_{\SSS, \sigma}/\mathfrak{D}_{\SSS, \sigma}$.  A direct computation gives that
\[
\dim(\mathfrak{Z}_{\SSS, \sigma}/\mathfrak{D}_{\SSS, \sigma}) = |\operatorname{rows}(\SSS) \sqcup \operatorname{cols}(\SSS)| = |\operatorname{CR}(\SSS)|.
\]
There is a fairly natural basis for this quotient which is indexed by $\operatorname{CR}(\SSS)$.  For $\RRR \in \operatorname{rows}(\SSS)$, let
\[
b_{\RRR} = e_{i, j} \quad\text{for $\dedge{i}{j} \in \RRR$ with $i$ minimal},
\]
and for $\CCC \in \operatorname{cols}(\SSS)$, let
\[
b_{\CCC} = \sum_{\edge{i}{j} \in \CCC} \left( \prod_{\substack{\edge{r}{s}, \edge{s+1}{t} \in \CCC \\ s+1 \le i}} \sigma \left( \binding{r}{s}{s+1}{t} \right) \right) e_{i, j} \qquad\text{with}\; \prod_{\emptyset} = 1,
\]
so that $b_{\CCC} \in \mathfrak{Z}_{\SSS, \sigma}$, $\operatorname{supp}(\{b_{\CCC}\}) = \CCC$, and $(b_{\CCC})_{i, j} = 1$ for the unique $\edge{i}{j} \in \CCC$ with $i$ minimal.  Graphically, for each $\edge{i}{j} \in \CCC$ the $\edge{i}{j}$-entry of $b_{\CCC}$ is the product of the labels above the edge $\edge{i}{j}$ in the column $\CCC$, and all other entries of $b_{\CCC}$ are zero.

\begin{ex}\label{ex:splicebasis}
Continuing from Example~\ref{ex:CandZ}, let $n = 8$, $q = 5$, and take $\SSS = \lambda \sqcup \nu$, $\sigma$, $\operatorname{cols}(\SSS) = \{\CCC_{1}, \CCC_{2}\}$, and $\operatorname{rows}(\SSS) = \{\RRR_{1}, \RRR_{1}\}$ as in
\begin{center}
\begin{tikzpicture}[scale = 1, baseline = -0.1cm]
\begin{scope}[xshift = 0cm, yshift = 1cm]
\draw[fill = black] (0, 0) circle (2pt) node[inner sep = 2pt] (1) {};
\draw[fill = black] (1, 0) circle (2pt) node[inner sep = 2pt] (2) {};
\end{scope}
\begin{scope}[xshift = 0cm, yshift = 0cm]
\draw[fill = black] (0, 0) circle (2pt) node[inner sep = 2pt] (3) {};
\draw[fill = black] (1, 0) circle (2pt) node[inner sep = 2pt] (4) {};
\draw[fill = black] (2, 0) circle (2pt) node[inner sep = 2pt] (6) {};
\end{scope}
\begin{scope}[xshift = 0cm, yshift = -1cm]
\draw[fill = black] (0, 0) circle (2pt) node[inner sep = 2pt] (5) {};
\draw[fill = black] (1, 0) circle (2pt) node[inner sep = 2pt] (7) {};
\draw[fill = black] (2, 0) circle (2pt) node[inner sep = 2pt] (8) {};
\end{scope}
\draw[thick, color = black] (1) -- (2);
\draw[thick, color = black] (3) -- (4);
\draw[thick, color = black] (4) -- (6);
\draw[thick, color = black] (5) -- (7);
\draw[thick, color = black] (7) -- (8);
\draw[color = specialgold, dashed, thick] (1)-- (3);
\draw[color = specialgold, dashed, thick] (2) -- (4);
\draw[color = specialgold, dashed, thick] (3) -- (5);
\draw[color = specialgold, dashed, thick] (4) -- (7);
\draw[color = specialgold, dashed, thick] (6) -- (8);
\path (1) -- node[pos = 0.5, color = red] {$3$} (4);
\path (3) -- node[pos = 0.5, color = red] {$1$} (7);
\path (4) -- node[pos = 0.5, color = red] {$2$} (8);
\node[above] at (1) {$\scriptstyle 1$};
\node[above] at (2) {$\scriptstyle 2$};
\node[above left] at (3) {$\scriptstyle 3$};
\node[above right] at (4) {$\scriptstyle 4$};
\node[below] at (5) {$\scriptstyle 5$};
\node[above right] at (6) {$\scriptstyle 6$};
\node[below] at (7) {$\scriptstyle 7$};
\node[below] at (8){$\scriptstyle 8$};
\begin{scope}[xshift = 3cm, yshift = 1.4cm]
\node at (0, -0.9) {$\RRR_{1}$};
\node at (0, -1.9) {$\RRR_{2}$};
\end{scope}
\begin{scope}[xshift = -0.5cm, yshift = 0.9cm]
\node at (1, 1) {$\CCC_{1}$};
\node at (2, 1) {$\CCC_{2}$};
\end{scope}
\end{tikzpicture}
.
\end{center}
Accordingly,
\[
b_{\CCC_{1}} = e_{1, 2} + 3e_{3, 4} + 3e_{5, 7},\quad b_{\CCC_{2}} = e_{4, 6} + 2e_{7, 8},\quad b_{\RRR_{1}} = e_{1, 3},\;\text{and}\quad b_{\RRR_{2}} = e_{3, 5}.
\]
Refer to Example~\ref{ex:CandZ} to verify that these elements, modulo $\mathfrak{D}_{\SSS, \sigma}$, give a basis of $\mathfrak{Z}_{\SSS, \sigma}/\mathfrak{D}_{\SSS, \sigma}$.
\end{ex}

\begin{lem}\label{lem:RCbasis}
For $(\SSS, \sigma) \in \mathscr{T}_{n}(q)$, the set
\[
\{b_{\CCC} + \mathfrak{D}_{\SSS, \sigma} \;|\; \CCC \in \operatorname{cols}(\SSS) \} \cup \{b_{\RRR} + \mathfrak{D}_{\SSS, \sigma}\;|\; \RRR \in \operatorname{rows}(\SSS)\} 
\]
is a basis for $\mathfrak{Z}_{\SSS, \sigma}/\mathfrak{D}_{\SSS, \sigma}$.
\end{lem}
\begin{proof}
Proceed in two steps, based on the division of $\mathfrak{Z}_{\SSS, \sigma}/\mathfrak{D}_{\SSS, \sigma}$ into the subquotients $\ut_{\upnum{\lambda}{1}}/\mathfrak{D}_{\SSS, \sigma}$ and $\mathfrak{Z}_{\SSS, \sigma}/\ut_{\upnum{\lambda}{1}}$.  First, if $\SSS = \lambda \sqcup \nu$, then
\[
 \bigoplus_{\edge{i}{j} \in \nu} \FF_{q}e_{i, j} = \bigoplus_{\RRR \in \operatorname{rows}(\SSS)} \FF_{q}b_{\RRR} \oplus \hspace{-0.5cm} \bigoplus_{\subbinding{i}{j}{j+1}{l}  \in \operatorname{bind}(\SSS)} \hspace{-0.4cm} \FF_{q} \left(e_{i, j+1} - \sigma\left(\binding{i}{j}{j+1}{l}\right)e_{j, l}  \right),
\]
so
\[
\ut_{\upnum{\lambda}{1}} = \ut_{\upnum{\lambda}{1}\setminus \nu} \oplus  \bigoplus_{\edge{i}{j} \in \nu} \FF_{q}e_{i, j} = \mathfrak{D}_{\SSS, \sigma} \oplus \bigoplus_{\RRR \in \operatorname{rows}(\SSS)} \FF_{q}b_{\RRR}.
\]
Describing a basis for $\mathfrak{Z}_{\SSS, \sigma}/\ut_{\upnum{\lambda}{1}}$ is more straightforward.  The subspace
\[
\left\{ \sum_{\edge{i}{j} \in \lambda} a_{i, j}e_{i, j} \;\middle|\; a_{i, j} = \sigma\left(\binding{i}{j}{j+1}{l}\right)a_{j+1, l} \;\text{ for each}\; \binding{i}{j}{j+1}{l} \in \operatorname{bind}(\SSS)  \right\} \subseteq \mathfrak{Z}_{\SSS, \sigma}
\]
is a transversal of $\mathfrak{Z}_{\SSS, \sigma}/\ut_{\upnum{\lambda}{1}}$, and the set $\{b_{\CCC} \;|\; \CCC \in \operatorname{cols}(\SSS)\}$ is a basis for this subspace.
\end{proof}

\begin{cor}
Let $\lambda \in \mathrm{NNSP}_{n}$ and $(\SSS, \sigma) \in \mathscr{T}_{\lambda}(q)$.  Then
\[
\dim(\mathfrak{D}_{\SSS, \sigma}) = |\hspace{-.35em}\upnum{\lambda}{1}|  - |\operatorname{rows}(\SSS)| \quad\text{and} \quad\dim(\mathfrak{Z}_{\SSS, \sigma}) = |\hspace{-.35em}\upnum{\lambda}{1}|  + |\operatorname{cols}(\SSS)|.
\]
\end{cor}

\subsection{Labeled loopless binary matroids and ideals}

In this section, fix $\lambda \in \mathrm{NNSP}_{n}$ and $(\SSS, \sigma) \in \mathscr{T}_{\lambda}(q)$.  Recall that
\[
\operatorname{CR}(\SSS) = \operatorname{cols}(\SSS) \sqcup \operatorname{rows}(\SSS).
\]
Section~\ref{sec:spliceordering} describes a total order on the set $\operatorname{CR}(\SSS)$.  With this order, the construction in Section~\ref{sec:matroids} describes the set $\overset{\to}{\mathscr{G}}_{\operatorname{CR}(\SSS)}(q)$ of $\FF_{q}^{\times}$-labeled loopless binary matroids on $\operatorname{CR}(\SSS)$.  

Each $(\MMM, \tau) \in \overset{\to}{\mathscr{G}}_{\operatorname{CR}(\SSS)}(q)$ will determine an ideal of $\ut_{n}$.  If $\MMM = (U, V, E)$, let
\begin{equation}\label{eq:idealdef}
\mathfrak{n}_{\SSS, \sigma, \MMM, \tau} = \FF_{q}\text{-span}\left\{ b_{\UUU} + \sum_{(\UUU, \VVV) \in E } \tau(\UUU, \VVV) b_{\VVV} \;\middle|\; \begin{array}{c} \text{$\UUU \in U$ with either} \\ \text{$(\UUU, \VVV) \in E$ for some} \\ \text{$\VVV \in V$,  or $\UUU \in \operatorname{cols}(\SSS)$} \end{array} \right\} + \mathfrak{D}_{\SSS, \sigma}.
\end{equation}
The subspace $\mathfrak{n}_{\SSS, \sigma, \MMM, \tau}$ satisfies $\mathfrak{D}_{\SSS, \sigma} \subseteq \mathfrak{n}_{\SSS, \sigma, \MMM, \tau} \subseteq \mathfrak{Z}_{\SSS, \sigma}$,  so by Lemma~\ref{lem:Zcommutator} is an ideal of $\ut_{n}$.

\begin{prop}\label{prop:matroidbijection}
For all $(\SSS, \sigma) \in \mathscr{T}_{n}(q)$, the map
\[
\begin{array}{rcl}
\overset{\to}{\mathscr{G}}_{\operatorname{CR}(\SSS)}(q) & \to & \operatorname{fam}(\SSS, \sigma) \\[0.25em]
(\MMM, \tau) & \mapsto & \mathfrak{n}_{\SSS, \sigma, \MMM, \tau}
\end{array}
\]
is a bijection.
\end{prop}

This proposition gives the remainder of Theorem~\ref{thm:ideals}.  The proof will follow Lemma~\ref{lem:reducedbasis}.

\begin{ex}\label{ex:matroidbijection}
For $n = 6$, and $q = 5$, take $\SSS = \lambda \sqcup \nu$, $\sigma: \operatorname{bind}(\SSS) \to \FF_{5}^{\times}$, $\operatorname{cols}(\SSS) = \{\CCC_{1}, \CCC_{2}\}$, and $\operatorname{rows}(\SSS) = \{\RRR_{1}\}$ as given by
\begin{center}
\begin{tikzpicture}[scale = 1, baseline = -0.6cm]
\begin{scope}[xshift = 0cm, yshift = 0cm]
\draw[thick, color = black] (0, 0) -- node[above, pos = 0.5] {} (1, 0);
\draw[color = specialgold, dashed, thick] (0, 0) -- (0, -1);
\draw[color = specialgold, dashed, thick] (1, 0) -- (1, -1);
\path (0.5, 0) -- node[pos = 0.5, color = red] {$3$} (0.5, -1);
\draw[fill = black] (0, 0) circle (2pt) node[above] {$\scriptstyle 1$};
\draw[fill = black] (1, 0) circle (2pt) node[above] {$\scriptstyle 2$};
\end{scope}
\begin{scope}[xshift = 0cm, yshift = -1cm]
\draw[thick, color = black] (0, 0) -- node[above, pos = 0.3] {} (1, 0);
\draw[fill = black] (0, 0) circle (2pt) node[below] {$\scriptstyle 3$};
\draw[fill = black] (1, 0) circle (2pt) node[below] {$\scriptstyle 5$};
\end{scope}
\begin{scope}[xshift = 2cm, yshift = -1cm]
\draw[thick, color = black] (0, 0) -- node[above, pos = 0.5] {} (1, 0);
\draw[fill = black] (0, 0) circle (2pt) node[below] {$\scriptstyle 4$};
\draw[fill = black] (1, 0) circle (2pt) node[below] {$\scriptstyle 6$};
\end{scope}
\begin{scope}[xshift = 4cm, yshift = 0.6cm]
\node at (0, -0.9) {$\RRR_{1}$};
\end{scope}
\begin{scope}[xshift = -0.5cm, yshift = 0cm]
\node at (1, 1) {$\CCC_{1}$};
\node at (3, 1) {$\CCC_{2}$};
\end{scope}
\end{tikzpicture}.
\end{center}
The canonical order on $\operatorname{CR}(\SSS)$ is $\CCC_{1} < \CCC_{2} < \RRR_{1}$, and the set of loopless binary matroids on an isomorphic ordered set is given in Example~\ref{ex:unibipartite}.  Below are three elements of $\overset{\to}{\mathscr{G}}_{\operatorname{CR}(\SSS)}(5)$ with the corresponding members of $\operatorname{fam}(\SSS, \sigma)$:
\begin{align*}
\begin{tikzpicture}[baseline = 0.5*0.75cm]
\draw[fill] (0, 0.75) circle (2pt) node[inner sep = 2pt] (T1) {};
\draw[fill] (0, 0) circle (2pt) node[inner sep = 2pt] (M) {};
\draw[fill] (1, 0.75) circle (2pt) node[inner sep = 2pt] (T2) {};
\node[left] at (T1) {$\CCC_{1}$};
\node[left] at (M) {$\CCC_{2}$};
\node[right] at (T2) {$\RRR_{1}$};
\draw[->] (M) -- node[pos = 0.6, below] {\color{red} $4$} (T2);
\end{tikzpicture}
\quad &\longmapsto \quad
\FF_{q} b_{\CCC_{1}} + \FF_{q}(b_{\CCC_{2}} + 4 b_{\RRR_{1}}) + \mathfrak{D}_{\SSS, \sigma}; \\[0.5em]
\begin{tikzpicture}[baseline = 0.5*0.75cm]
\draw[fill] (0, 0.75) circle (2pt) node[inner sep = 2pt] (T1) {};
\draw[fill] (0, 0) circle (2pt) node[inner sep = 2pt] (M) {};
\draw[fill] (1, 0.75) circle (2pt) node[inner sep = 2pt] (T2) {};
\node[left] at (T1) {$\CCC_{1}$};
\node[left] at (M) {$\CCC_{2}$};
\node[right] at (T2) {$\RRR_{1}$};
\draw[->] (T1) -- node[pos = 0.5, above] {\color{red} $1$} (T2);
\draw[->] (M) -- node[pos = 0.6, below] {\color{red} $2$} (T2);
\end{tikzpicture}
\quad &\longmapsto \quad
\FF_{q} (b_{\CCC_{1}} + b_{\RRR_{1}}) + \FF_{q}(b_{\CCC_{2}} + 2 b_{\RRR_{1}}) + \mathfrak{D}_{\SSS, \sigma}\text{; and}\\[0.5em]
\begin{tikzpicture}[baseline = 0.5*0.75cm]
\draw[fill] (0, 0.75) circle (2pt) node[inner sep = 2pt] (T1) {};
\draw[fill] (1, 0.75) circle (2pt) node[inner sep = 2pt] (M) {};
\draw[fill] (0, 0) circle (2pt) node[inner sep = 2pt] (T2) {};
\node[left] at (T1) {$\CCC_{1}$};
\node[right] at (M) {$\CCC_{2}$};
\node[left] at (T2) {$\RRR_{1}$};
\draw[->] (T1) -- node[pos = 0.5, above] {\color{red} $3$} (M);
\end{tikzpicture}
\quad &\longmapsto \quad
\FF_{q} (b_{\CCC_{1}} + 3b_{\CCC_{2}}) + \mathfrak{D}_{\SSS, \sigma}.
\end{align*}
\end{ex}

In order to prove Proposition~\ref{prop:matroidbijection}, we will view $\mathfrak{Z}_{\SSS, \sigma}/\mathfrak{D}_{\SSS, \sigma}$ as the set of $\operatorname{CR}(\SSS)$-indexed coordinates from $\FF_{q}$, as in
\[
(\alpha_{\CCC_{1}}, \ldots, \alpha_{\CCC_{|\operatorname{cols}(\SSS)|}}, \alpha_{\RRR_{1}}, \ldots, \alpha_{\RRR_{|\operatorname{rows}(\SSS)|}} ) = \sum_{\VVV \in \operatorname{CR}(\SSS)} \alpha_{\VVV} b_{\VVV} + \mathfrak{D}_{\SSS, \sigma}.
\]
Given any subspace $W$ of $\mathfrak{Z}_{\SSS, \sigma}/\mathfrak{D}_{\SSS, \sigma}$, there is a canonical basis for $W$ which will serve as a stand-in for $W$ itself.  Let $k = \dim(W)$ and take $(x_{i, \VVV})_{1 \le i \le k, \VVV \in \operatorname{CR}(\SSS)}$ to be the unique reduced row-echelon form matrix over $\FF_{q}$ with row span $W$.  That is, if
\[
x_{i} = (x_{i, \CCC_{1}}, \ldots, x_{i, \RRR_{|\operatorname{rows}(\SSS)|}}) \qquad\text{for $1 \le i \le k$,}
\]
then $W = \FF_{q}\text{-span}\{x_{1}, \ldots, x_{k}\}$ and there are elements $\UUU_{1} < \UUU_{2} < \cdots < \UUU_{k}$ of $\operatorname{CR}(\SSS)$ with
\begin{enumerate}
\item $x_{i, \VVV} = 0$ for all $\VVV < \UUU_{i}$, and 

\item $x_{i,\, \UUU_{j}} = \delta_{i, j}$.  

\end{enumerate}
Call this basis the \emph{RRE basis} of $W$.

\begin{lem}\label{lem:reducedbasis}
Let $W$ be a subspace of $\mathfrak{Z}_{\SSS, \sigma}/\mathfrak{D}_{\SSS, \sigma}$ with RRE basis $\{x_{1}, \ldots, x_{k}\}$.  Then $W = \mathfrak{n}/\mathfrak{D}_{\SSS, \sigma}$ for some $\mathfrak{n} \in \operatorname{fam}(\SSS, \sigma)$ if and only if
\begin{enumerate}
\item[(a)] for each $\CCC \in \operatorname{cols}(\SSS)$, $x_{i, \CCC} \neq 0$ for at least one value of $i$, $1 \le i \le k$, and 

\item[(b)] for each $1 \le i \le k$, $x_{i} \neq b_{\RRR} + \mathfrak{D}_{\SSS, \sigma}$ for any $\RRR \in \operatorname{rows}(\SSS)$.

\end{enumerate}
\end{lem}
\begin{proof}
Let $\mathfrak{n}$ be the unique subspace of $\mathfrak{Z}_{\SSS, \sigma}$ with $\mathfrak{D}_{\SSS, \sigma} \subseteq \mathfrak{n}$ and $\mathfrak{n}/\mathfrak{D}_{\SSS, \sigma} = W$.  By Lemma~\ref{lem:Zcommutator}, $\mathfrak{n}$ is an ideal of $\ut_{n}$.  Recall that $\mathfrak{n} \in \operatorname{fam}(\SSS, \sigma)$ if and only if
\[
\begin{array}{rl}
\text{(i)} &\quad \operatorname{supp}(\mathfrak{n}) = \up{\lambda}, \\[0.25em]
\text{(ii)} &\quad \text{for each $\edge{i}{j} \in \upnum{\lambda}{1}$, $e_{i, j} \in \mathfrak{n}$ if and only if $\edge{i}{j} \notin \nu$, and}\\[0.25em]
\text{(iii)} &\quad \mathfrak{n} \subseteq \mathfrak{Z}_{\SSS, \sigma}.
\end{array}
\]
By assumption, (iii) holds.  The following shows that (i) is equivalent to (a), and (ii) to (b).  

For each $\edge{r}{s} \in \lambda$, the $(r, s)$-entry of each member of any coset in $\ut_{\lambda}/\ut_{\upnum{\lambda}{1}}$ will be identical.  As $\mathfrak{D}_{\SSS, \sigma} \subseteq \ut_{\upnum{\lambda}{1}}$, the same is true for $\ut_{\lambda}/\mathfrak{D}_{\SSS, \sigma}$.  Therefore, for each $a \in \mathfrak{Z}_{\SSS, \sigma}$, $\CCC \in \operatorname{cols}(\SSS)$, and $\edge{r}{s} \in \CCC$, the $(r, s)$-entry of each member of $a + \mathfrak{D}_{\SSS, \sigma}$ is nonzero if and only if the $\CCC$-coordinate of $a + \mathfrak{D}_{\SSS, \sigma}$ is nonzero.  Since $\operatorname{cols}(\SSS)$ partitions the set $\lambda$, (i) is equivalent to (a).

For $\RRR \in \operatorname{rows}(\SSS)$ and $\dedge{i}{j} \in \RRR$, 
\[
\FF_{q}e_{i, j} + \mathfrak{D}_{\SSS, \sigma} = \bigoplus_{\dedge{r}{s} \in \RRR} \FF_{q}e_{r, s} + \mathfrak{D}_{\SSS, \sigma} = \FF_{q}b_{\RRR} + \mathfrak{D}_{\SSS, \sigma},
\]
so (ii) holds if and only if $b_{\RRR} + \mathfrak{D}_{\SSS, \sigma} \notin W$ for any $\RRR \in \operatorname{rows}(\SSS)$.  With the properties of an RRE basis, this is equivalent to condition (b).
\end{proof}

\begin{proof}[Proof of Proposition~\ref{prop:matroidbijection}]
It is not clear from the outset that each $\mathfrak{n}_{\SSS, \sigma, \MMM, \tau}$ belongs to $\operatorname{fam}(\SSS, \sigma)$, so this must be established.  Fix  $(\MMM, \tau) \in \overset{\to}{\mathscr{G}}_{\operatorname{CR}(\SSS)}(q)$ with $\MMM = (U, V, E)$ and write
\[
\{\UUU_{1}, \UUU_{2}, \ldots, \UUU_{k}\} = \left\{ \begin{array}{c} \text{$\UUU \in U$ with either} \\ \text{$(\UUU, \VVV) \in E$ for some} \\ \text{$\VVV \in V$,  or $\UUU \in \operatorname{cols}(\SSS)$} \end{array} \right\}
\]
in increasing order.  Let
\[
x_{i} = b_{\UUU_{i}} + \sum_{(\UUU_{i}, \VVV) \in E} \tau(\UUU_{i}, \VVV) b_{\VVV}
\]
so that by definition
\[
\mathfrak{n}_{\SSS, \sigma, \MMM, \tau} =  \FF_{q}\text{-span}\{x_{i} \;|\; 1 \le i \le k\} + \mathfrak{D}_{\SSS, \sigma}.
\]
The set $\{x_{i} \;|\; 1 \le i \le k\}$ is the RRE basis of $W = \mathfrak{n}_{\SSS, \sigma, \MMM, \tau}/\mathfrak{D}_{\SSS, \sigma}$, as the definition of $\overset{\to}{\mathscr{G}}_{\operatorname{CR}(\SSS)}$ ensures that 
\begin{enumerate}
\item $x_{i, \VVV} = 0$ for all $\VVV < \UUU_{i}$, and 

\item $x_{i, \UUU_{j}} = \delta_{i, j}$.  

\end{enumerate}
This basis satisfies the conditions of Lemma~\ref{lem:reducedbasis}, so $\mathfrak{n} \in \operatorname{fam}(\SSS, \sigma)$.

To show that the given map is also a bijection, construct an inverse as follows.  For $\mathfrak{n} \in \operatorname{fam}(\SSS, \sigma)$, let $W = \mathfrak{n}/\mathfrak{D}_{\SSS, \sigma}$ and take $\{x_{i} \;|\; 1 \le i \le k\}$ to be the RRE basis of $W$ with corresponding indices $\UUU_{1} < \UUU_{2} < \cdots < \UUU_{k}$ satisfying conditions 1.~and 2.~above.  Let
\begin{align*}
V &= \{\VVV \in \operatorname{CR}(\SSS) \;|\;\text{$x_{i, \VVV} \neq 0$ and $\VVV \neq \UUU_{i}$ for every $1 \le i \le k$}\}, \\
U &= \operatorname{CR}(\SSS) \setminus V, \\
E &= \{(\UUU_{i}, \VVV) \in U \times V \;|\; x_{i, \VVV} \neq 0 \},\;\text{and} \\
\tau(\UUU_{i}, \VVV) &= x_{i, \VVV}, 
\end{align*}
and let $\MMM = (U, V, E)$.   Then $\MMM \in \overset{\to}{\mathscr{G}}_{\operatorname{CR}(\SSS)}$ and $(\MMM, \tau) \in \overset{\to}{\mathscr{G}}_{\operatorname{CR}(\SSS)}(q)$, since $\{x_{i} \;|\; 1 \le i \le k\}$ is an RRE basis.  Finally, $\mathfrak{n}_{\SSS, \sigma, \MMM, \tau} = \mathfrak{n}$ by definition.
\end{proof}

\section{Normal subgroups}
\label{sec:normalsubgroups}

Recall that Theorem~\ref{introthm:A} asserts a correspondence between the normal subgroups of $\UT_{n}(\FF_{q})$ and certain ideal-like additive subgroups of $\ut_{n}(\FF_{q})$, and that as a consequence the results of Section~\ref{sec:ideals} also apply to certain normal subgroups, giving Theorem~\ref{introthm:B}.  One particularly nice consequence is the following corollary.

\begin{cor}\label{cor:normalsubgroups}
Let $p$ be a prime.  The map
\[
\begin{array}{rcl}
\left\{(\SSS, \sigma, \MMM, \tau) \;\middle|\; (\SSS, \sigma) \in \mathscr{T}_{n}(p), (\MMM, \tau) \in \overset{\to}{\mathscr{G}}_{\operatorname{CR}(\SSS)}(p)  \right\} & \longrightarrow & \{ N \trianglelefteq \UT_{n}(\FF_{p}) \} \\
(\SSS, \sigma, \MMM, \tau) & \longmapsto & 1 + \mathfrak{n}_{\SSS, \sigma, \MMM, \tau}
\end{array}
\]
is a bijection.
\end{cor}

A number of examples which apply Theorem~\ref{thm:ideals} to construct previously obscure ideals (and thus normal subgroups) can be found in Section~\ref{sec:ideals}.  The indexing for the better known normal subgroups of $\UT_{n}(\FF_{q})$ in Corollary~\ref{cor:normalsubgroups} is similar but distinct from previous work.

In~\cite[Theorem 4.1]{Marberg11b}, Marberg constructs a family of normal subgroups of $\UT_{n}(\FF_{p})$ which have the form $1 + \mathfrak{a}$ for a two-sided associative algebra ideal $\mathfrak{a} \subseteq \ut_{n}(\FF_{p})$.  These ideals are nicely characterized as the subspaces $\mathfrak{a} \subseteq \ut_{n}(\FF_{p})$ with 
\[
\ut_{\upnum{\lambda}{1}} \subseteq \mathfrak{a} \subseteq \ut_{\lambda},
\]
where $\lambda$ is the nonnesting set partition with $\operatorname{supp}(\mathfrak{a}) = \up{\lambda}$.  Thus in Corollary~\ref{cor:normalsubgroups}, Marberg's normal subgroups---including all central subgroups and normal pattern subgroups---are indexed by trivial splices.  For example,
\[
\begin{pmatrix} 1 & \alpha & \ast & \ast \\ 
0 & 1 & 0 & \ast \\ 
0 & 0 & 1 & 2\alpha \\ 
0 & 0 & 0 & 1 \end{pmatrix}
\;\longleftrightarrow\;
(\SSS, \sigma, \MMM, \tau)
\quad\text{with}\quad
\begin{array}{c}
(\SSS, \sigma) = 
\begin{tikzpicture}[scale = 1, baseline = -1.1cm]
\begin{scope}[xshift = 0cm, yshift = -1cm]
\draw[thick, color = black] (0, 0) -- node[above, pos = 0.3] {} (1, 0);
\draw[fill = black] (0, 0) circle (2pt) node[below] {$\scriptstyle 1$};
\draw[fill = black] (1, 0) circle (2pt) node[below] {$\scriptstyle 2$};
\end{scope}
\begin{scope}[xshift = 2cm, yshift = -1cm]
\draw[thick, color = black] (0, 0) -- node[above, pos = 0.5] {} (1, 0);
\draw[fill = black] (0, 0) circle (2pt) node[below] {$\scriptstyle 3$};
\draw[fill = black] (1, 0) circle (2pt) node[below] {$\scriptstyle 4$};
\end{scope}
\begin{scope}[xshift = -0.5cm, yshift = 0cm]
\node at (1, -0.25) {$\CCC_{1}$};
\node at (3, -0.25) {$\CCC_{2}$};
\end{scope}
\end{tikzpicture}
,\\[1em]
\text{and}\;\;(\MMM, \tau) = 
\begin{tikzpicture}[baseline = 0.65cm]
\draw[fill] (0, 0.75) circle (2pt) node[inner sep = 2pt] (T1) {};
\draw[fill] (1, 0.75) circle (2pt) node[inner sep = 2pt] (M) {};
\node[left] at (T1) {$\CCC_{1}$};
\node[right] at (M) {$\CCC_{2}$};
\draw[->] (T1) -- node[pos = 0.5, above] {\color{red} $2$} (M);
\end{tikzpicture}. \\[1.5em]
\end{array}
\]
More generally, these subgroups are indexed by tuples $(\lambda, f_{\emptyset}, \MMM, \tau)$, where $\lambda = \lambda \sqcup \emptyset$ denotes the trivial splice of $\lambda \in \mathrm{NNSP}_{n}$, $f_{\emptyset}$ is the unique labeling of $\operatorname{bind}(\lambda) = \emptyset$, and $(\MMM, \tau)$ is an $\FF_{p}^{\times}$-labeled loopless binary matroid on the set $\operatorname{CR}(\lambda)$, which contains only one-element columns.

Some caution is required when comparing Corollary~\ref{cor:normalsubgroups} to~\cite[Theorem 4.1]{Marberg11b}, however, as the latter uses a different convention for relating nonnesting set partitions to subsets of $\ut_{n}(\FF_{q})$.  In this paper, a nonnesting set partition determines the support of each corresponding subgroup, but in~\cite{Marberg11b} nonnesting set partitions determine a different property which is, in a sense, dual to support.  The dual property is difficult to describe in the scope of this paper, but the indexing schemes are morally equivalent (see~\cite[Section 3.1]{AliniaeifardThiem20} for details), and versions of both results can be given in either convention.

The remainder of the section is divided into two subsections.  Subsection~\ref{sec:correspondenceproof} gives a proof of Theorem~\ref{introthm:A}, which is not significantly different from that of~\cite[Theorem 1]{Levcuk74}. This is included mainly for completeness, but some intermediate results are needed in the sequel.  Subsection~\ref{sec:normalsubgroupcorollaries} states new results about the normal subgroups of $\UT_{n}(\FF_{q})$, including an enumerative formula and a discussion of the lattice of normal subgroups.

\subsection{Proof of Theorem~\ref{thm:correspondence}}
\label{sec:correspondenceproof}

In this section fix a prime power $q$, and write $\ut_{n} = \ut_{n}(\FF_{q})$ and $\UT_{n} = \UT_{n}(\FF_{q})$.  In order to distinguish the Lie bracket from the group commutator, I will adopt the notation
\[
[a, b]_{\operatorname{Lie}} = ab - ba \quad\text{and}\quad [a, b] = a^{-1}b^{-1}ab.
\]

The following computation will be used repeatedly: for $a \in \UT_{n}$, $\edge{s}{t} \in [[n]]$, and $x \in \FF_{q}$, 
\begin{align}\label{eq:grpcommutator}
[1 + xe_{s, t}, a] &= 1 + x\sum_{\edge{r}{u} \succ \edge{s}{t}} (a^{-1})_{r, s}a_{t, u}e_{r, u}.
\end{align}

\begin{lem}\label{lem:subgroupupnum2}
Let $N \trianglelefteq \UT_{n}$ and take $\lambda \in \mathrm{NNSP}_{n}$ with $\operatorname{supp}(N) = \up{\lambda}$.  Then
\[
[\UT_{n}, [\UT_{n}, N]] = \UT_{\upnum{\lambda}{2}}.
\]
\end{lem}
\begin{proof}
By assumption $N \subseteq \UT_{\lambda}$, so $[\UT_{n}, [\UT_{n}, N]] \subseteq \UT_{\upnum{\lambda}{2}}$.  The opposite containment will follow from the fact that 
\[
\{1 + xe_{i, l} \;|\; \edge{i}{l} \in \upnum{\lambda}{2}, x \in \FF_{q}\} \subseteq [\UT_{n}, [\UT_{n}, N]],
\]
which generates $\UT_{\upnum{\lambda}{2}}$.  Fix $\edge{i}{l} \in \upnum{\lambda}{2}$ and $x \in \FF_{q}$.  It is possible to choose
\begin{itemize}
\item $\edge{j}{k} \in \lambda$ with $\edge{j}{k} \preceq \edge{i}{l}$ and $ \Ht(\edge{j}{k}) + 2 \le \Ht(\edge{i}{l})$ and 

\item $a \in N$ with $a_{j, k} \neq 0$.

\end{itemize}
Use this choice to define a new element
\[
b = \begin{cases} 
[1 + xe_{k+1, l}, ([1 + e_{k, k+1} , a^{-1}])^{-1}] = 1 + x \sum_{r < k} a_{r, k}e_{r, l} & \text{if $i = j < k < l$}, \\
[1 + xe_{i, j}, [1 + e_{k, l}, a^{-1}]] =  1 + x\sum_{u \ge l} a_{j, k}(a^{-1})_{l, u}  e_{i, u} & \text{if $i < j < k < l$}, \\
[1 + xe_{i, j-1} , [1 + e_{j-1, j} , a]] =  1 + x\sum_{u > j} a_{j, u}e_{i, u} & \text{if $i < j < k = l$.}
\end{cases}
\]
The $(i, l)$-entry of $b$ is $xa_{j, k} \neq 0$, so a standard computation (e.g.~\cite[Corollary 3.4]{DiaconisIsaacs}) shows that $b$ is conjugate to $1 + xa_{j, k}e_{i, l}$.
\end{proof}

\begin{lem}\label{lem:commutatorcorrespondence}
Let $N$ be a subgroup of $\UT_{n}$.  For any $\lambda \in \mathrm{NNSP}_{n}$ with $\operatorname{supp}(N) \subseteq \up{\lambda}$,
\[
[\UT_{n}, N]\UT_{\upnum{\lambda}{2}} = 1 + \langle[b, a -1 ]_{\operatorname{Lie}} \;|\; b \in \ut_{n}, a \in N \rangle + \ut_{\upnum{\lambda}{2}}.
\]
\end{lem}
\begin{proof}
For $a \in \UT_{\lambda}$, $1 \le j < n$, and $x \in \FF_{q}$ equation~\eqref{eq:grpcommutator} gives
\[
[1 + xe_{j, j+1}, a]\UT_{\upnum{\lambda}{2}} = \begin{cases} (1 + x((a^{-1})_{i, j}e_{i, j+1} + a_{j+1, l}e_{j, l}))\UT_{\upnum{\lambda}{2}} &\begin{array}{l}  \text{if $\edge{i}{j}, \edge{j+1}{l} \in \lambda$,} \end{array} \\
(1 + x(a^{-1})_{i, j}e_{i, j+1})\UT_{\upnum{\lambda}{2}} &\begin{array}{l}  \text{if $\edge{i}{j} \in \lambda$ and for all} \\ \quad \text{$l \in [n]$, $\edge{j+1}{l} \notin \lambda$,}\end{array}\\
(1 + x(a^{-1})_{j+1, l}e_{j, l})\UT_{\upnum{\lambda}{2}} &\begin{array}{l}  \text{if $\edge{j+1}{l} \in \lambda$ and for} \\ \quad \text{all $i \in [n]$, $\edge{i}{j} \notin \lambda$,} \end{array} \\
\UT_{\upnum{\lambda}{2}} & \begin{array}{l} \text{otherwise.}\end{array} \end{cases}
\]
As $\operatorname{supp}(\{a\}) \subseteq \up{\lambda}$, $(a^{-1})_{i, j} = - a_{i, j}$ for each $\edge{i}{j} \in \lambda$.  Therefore the above expression matches the analogous statement about commutators from Lemma~\ref{lem:commutator} and
\[
[1 + xe_{j, j+1}, a]\UT_{\upnum{\lambda}{2}} = 1 + [xe_{j, j+1}, a - 1]_{\operatorname{Lie}} + \ut_{\upnum{\lambda}{2}}
\]

By assumption $[\UT_{n}, N] \subseteq \UT_{\upnum{\lambda}{1}}$, so $[\UT_{n}, N]\UT_{\upnum{\lambda}{2}}/\UT_{\upnum{\lambda}{2}}$ is central in $\UT_{n}/\UT_{\upnum{\lambda}{2}}$.  Therefore, for any $a \in N$, $1 \le j, k < n$, and $x, y \in \FF_{q}$
\[
[(1 + xe_{j, j+1})(1 + ye_{r, r+1}), a]\UT_{\upnum{\lambda}{2}} = [1 + xe_{j, j+1}, a]\UT_{\upnum{\lambda}{2}}[1 + ye_{r, r+1}, a]\UT_{\upnum{\lambda}{2}}.
\]
Elements of the form $1 + xe_{j, j+1}$ generate $\UT_{n}$, and from consideration of support
\[
\begin{array}{rcl}
\UT_{\upnum{\lambda}{1}}/\UT_{\upnum{\lambda}{2}} & \longrightarrow & \ut_{\upnum{\lambda}{1}}/\ut_{\upnum{\lambda}{2}} \\
a\UT_{\upnum{\lambda}{2}} & \longmapsto & (a-1) + \ut_{\upnum{\lambda}{2}}
\end{array}
\]
is an isomorphism, completing the proof.
\end{proof}

\begin{lem}\label{lem:quotientisomorphism}
Let $\lambda \in \mathrm{NNSP}_{n}$ and $\SSS \in \mathscr{T}_{\lambda}$.  The bijection
\[
\begin{array}{rcl}
\ut_{\lambda}/\ut_{\up{\lambda} \setminus \SSS} & \longrightarrow &  \UT_{\lambda}/\UT_{\up{\lambda} \setminus \SSS} \\
a + \ut_{\up{\lambda} \setminus \SSS} & \longmapsto & (1+a) \UT_{\up{\lambda} \setminus \SSS}
\end{array}
\]
is an isomorphism of groups.
\end{lem}
\begin{proof}
By consideration of support the statement will follow if
\[
\Big\{\edge{i}{k} \;\Big|\; \edge{i}{j}, \edge{j}{k} \in \up{\lambda}\Big\} \subseteq \up{\lambda} \setminus \SSS.
\]

Suppose that $\edge{i}{j}, \edge{j}{k} \in \up{\lambda}$, and note that $\Ht(\edge{i}{k}) = \Ht(\edge{i}{j}) + \Ht(\edge{j}{k})$.  If either $ \Ht(\edge{i}{j}) > 1$ or $\Ht(\edge{j}{k}) > 1$, then $\edge{i}{k} \in \upnum{\lambda}{2}$, a subset of $\up{\lambda} \setminus \SSS$.  Otherwise $\edge{i}{j}, \edge{j}{k} \in \lambda$, in which case $i$ and $k$ are in the same connected component of the graph of $\lambda$.  In this case $\edge{i}{k}$ is not contained in any splice of $\lambda$ by Lemma~\ref{lem:spliceproperties}, so $\edge{i}{k} \in \up{\lambda} \setminus \SSS$.
\end{proof}

\begin{lem}\label{lem:additiveclosure}
Let $\mathfrak{n}$ be a subgroup of $\ut_{n}$ with $\{ [a, b]_{\operatorname{Lie}} \;|\; a \in \ut_{n}, b \in \mathfrak{n} \} \subseteq \mathfrak{n}$, and let $\overline{\mathfrak{n}} = \FF_{q}\operatorname{-span}(\mathfrak{n})$.  Then
\begin{enumerate}
\item $[\ut_{n}, \overline{\mathfrak{n}}]_{\operatorname{Lie}} = \langle [a, b]_{\operatorname{Lie}} \;|\; a \in \ut_{n}, b \in \mathfrak{n} \rangle$, and this is a subset of $\mathfrak{n}$, and

\item $\overline{\mathfrak{n}}$ is an ideal of $\ut_{n}$.

\end{enumerate}
\end{lem}
\begin{proof}
Take $a \in \ut_{n}$, $b_{1}, b_{2}, \ldots, b_{k} \in \mathfrak{n}$, and $\alpha_{1}, \alpha_{2}, \ldots, \alpha_{k} \in \FF_{q}$.  Then
\[
\left[a, \sum_{i = 1}^{k} \alpha_{i}b_{i} \right]_{\operatorname{Lie}} = \sum_{i =1}^{k} \alpha_{i}[a, b_{i}]_{\operatorname{Lie}} = \sum_{i =1}^{k} [\alpha_{i}a, b_{i}]_{\operatorname{Lie}} \subseteq \mathfrak{n} \subseteq \overline{\mathfrak{n}}. \qedhere 
\]
\end{proof}

\begin{proof}[Proof of Theorem~\ref{thm:correspondence}]
Let $N \subseteq \UT_{n}$ and $\mathfrak{n} \subseteq \ut_{n}$ with $N = 1 + \mathfrak{n}$, and take $\lambda \in \mathrm{NNSP}_{n}$ with $\up{\lambda} = \operatorname{supp}(N) = \operatorname{supp}(\mathfrak{n})$.  The aim is to show that $N$ is a normal subgroup if and only if $\mathfrak{n}$ is an additive subgroup with $\{ [a, b]_{\operatorname{Lie}} \;|\; a \in \ut_{n}, b \in \mathfrak{n} \} \subseteq \mathfrak{n}$.

First suppose that $N \trianglelefteq \UT_{n}$.  Let $\mathfrak{m} = [\UT_{n}, N] - 1$.  By Lemma~\ref{lem:subgroupupnum2} and Lemma~\ref{lem:commutatorcorrespondence}, 
\begin{equation}\label{eq:correspondencetheoremproof1}
\mathfrak{m} = \langle [a, b]_{\operatorname{Lie}} \;|\; a \in \ut_{n}, b \in \mathfrak{n} \rangle + \ut_{\upnum{\lambda}{2}},
\end{equation}
and $\mathfrak{m} \subseteq \mathfrak{n}$, so what remains is to show that $\mathfrak{n}$ is an additive group.  By equation~\eqref{eq:correspondencetheoremproof1}, $\mathfrak{m}$ and $\mathfrak{n}$ meet the conditions of Lemma~\ref{lem:splicecondition}, so there is a tight splice $\SSS$ of $\lambda$ for which 
\[
\ut_{\up{\lambda} \setminus \SSS} \subseteq \mathfrak{m} \subseteq \mathfrak{n}.
\]
Under the isomorphism of Lemma~\ref{lem:quotientisomorphism}, $\mathfrak{n}/\ut_{\up{\lambda} \setminus \SSS}$ maps to $N/\UT_{\up{\lambda} \setminus \SSS}$, so $\mathfrak{n}$ is a group.

Now suppose that $\mathfrak{n}$ is a subgroup of $\ut_{n}$ with $\{[a, b]_{\operatorname{Lie}} \;|\; a \in \ut_{n}, b \in \mathfrak{n} \} \subseteq \mathfrak{n}$.  By Lemma~\ref{lem:additiveclosure}, $\overline{\mathfrak{n}} = \FF_{q}\operatorname{-span}(\mathfrak{n})$ is an ideal and $[\ut_{n}, \overline{\mathfrak{n}}]_{\operatorname{Lie}} \subseteq \mathfrak{n}$.  By Lemma~\ref{lem:splicecondition}, there is a tight splice $\SSS$ of $\lambda$ for which
\[
\ut_{\up{\lambda} \setminus \SSS} \subseteq [\ut_{n}, \overline{\mathfrak{n}}]_{\operatorname{Lie}} \subseteq \mathfrak{n}.
\]
Under the isomorphism of Lemma~\ref{lem:quotientisomorphism}, $N/\UT_{\up{\lambda} \setminus \SSS}$ is the image of $\mathfrak{n}/\ut_{\up{\lambda} \setminus \SSS}$, and so $N$ is a subgroup of $\UT_{n}$.  Lemma~\ref{lem:commutatorcorrespondence} then gives
\[
[\UT_{n}, N] \subseteq \langle [a, b] \;|\; a \in \ut_{n}, b \in \mathfrak{n} \rangle + \ut_{\upnum{\lambda}{2}} \subseteq N,
\]
and so $N \trianglelefteq \UT_{n}$.
\end{proof}

A slightly stronger restatement of Lemma~\ref{lem:commutatorcorrespondence} now follows.

\begin{cor}\label{cor:commutatorcorrespondence}
Let $1 + \mathfrak{n} \trianglelefteq \UT_{n}(\FF_{q})$ and  $\overline{\mathfrak{n}} = \FF_{q}\operatorname{-span}(\mathfrak{n})$.  Then
\[
[\UT_{n}(\FF_{q}), 1 + \mathfrak{n}] = 1 + [\ut_{n}(\FF_{q}), \overline{\mathfrak{n}}]_{\operatorname{Lie}}.
\]
\end{cor}
\begin{proof}
This follows from Lemma~\ref{lem:subgroupupnum2}, Lemma~\ref{lem:commutatorcorrespondence}, and Theorem~\ref{thm:correspondence}.
\end{proof}

\subsection{Further results on normal subgroups}
\label{sec:normalsubgroupcorollaries}

Some aspects of my description of Lie algebra ideals extend to the set $\{N \trianglelefteq \UT_{n}(\FF_{q})\}$, even when $q$ is not prime.  For $(\SSS, \sigma) \in \mathscr{T}_{n}(q)$ with $\SSS = \lambda \sqcup \nu$, define the \emph{normal subgroup family}
\[
\operatorname{NSGfam}(\SSS, \sigma) = \left\{N \trianglelefteq \UT_{n}(\FF_{q}) \;\middle|\;  \begin{array}{c}
\text{$\operatorname{supp}(N) = \up{\lambda}$,} \\
\text{$\{{\textstyle \edge{i}{j}} \in \upnum{\lambda}{1} \;|\; 1 + \FF_{q}e_{i, j} \not\subseteq N\} = \nu$,}\\
\text{and $N \subseteq 1 + \mathfrak{Z}_{\SSS, \sigma}$} 
\end{array}\right\}.
\]
This set of subgroups relates to the previously defined $\operatorname{fam}(\SSS, \sigma)$ of \textit{ideals} by
\[
\{1 + \mathfrak{n} \;|\; \mathfrak{n} \in \operatorname{fam}(\SSS, \sigma)\} = \{ 1 + \mathfrak{n} \in \operatorname{NSGfam}(\SSS, \sigma) \;|\; \text{$\mathfrak{n}$ is an $\FF_{q}$-subspace}\}.
\]

\begin{cor}\label{cor:grpsplicecondition}
Each normal subgroup of $\UT_{n}(\FF_{q})$ is contained in a unique normal subgroup family:
\[
\{N \trianglelefteq \UT_{n}(\FF_{q})\} = \bigsqcup_{(\SSS, \sigma) \in \mathscr{T}_{n}(q)} \operatorname{NSGfam}(\SSS, \sigma).
\]
Further, $1 + \mathfrak{D}_{\SSS, \sigma} \le N$ for every $N \in \operatorname{NSGfam}(\SSS, \sigma)$.
\end{cor}
\begin{proof}
Let $N \trianglelefteq \UT_{n}(\FF_{q})$ with $\operatorname{supp}(N) = \up{\lambda}$ for $\lambda \in \mathrm{NNSP}_{n}$.  Then $\UT_{\upnum{\lambda}{2}} \subseteq [\UT_{n}(\FF_{q}), N]$, so Lemma~\ref{lem:commutatorcorrespondence} implies
\[
[\ut_{n}(\FF_{q}), N-1] + \ut_{\upnum{\lambda}{2}} \subseteq (N-1) \cap \ut_{\upnum{\lambda}{1}}.
\]
Thus Lemma~\ref{lem:splicecondition} applies to $\mathfrak{n} = N - 1$ and $\mathfrak{m} = \mathfrak{n} \cap \ut_{\upnum{\lambda}{1}}$, giving the first claim exactly.  The second claim follows from Lemma~\ref{lem:Zcommutator}, Proposition~\ref{prop:CandZ}, and Corollary~\ref{cor:commutatorcorrespondence}: 
\[
1 + \mathfrak{D}_{\SSS, \sigma} = [\UT_{n}(\FF_{q}), 1 + \mathfrak{Z}_{\SSS, \sigma}] \subseteq N. \qedhere
\]
\end{proof}

Lemma~\ref{lem:RCbasis} states that for any $(\SSS, \sigma) \in \mathscr{T}_{n}(q)$, 
\[
\mathfrak{Z}_{\SSS, \sigma}/\mathfrak{D}_{\SSS, \sigma} = \FF_{q}\operatorname{-span}\{b_{\VVV} + \mathfrak{D}_{\SSS, \sigma} \;|\; \VVV \in \operatorname{CR}(\SSS)\}.
\]

\begin{lem}\label{lem:intermediatesubgroup}
Let $(\SSS, \sigma) \in \mathscr{T}_{n}(q)$, and take $W$ to be an additive subgroup of $\mathfrak{Z}_{\SSS, \sigma}/\mathfrak{D}_{\SSS, \sigma}$.  Then $W  = \mathfrak{n}/\mathfrak{D}_{\SSS, \sigma}$ for a normal subgroup $1 + \mathfrak{n} \in \operatorname{NSGfam}(\SSS, \sigma)$ if and only if $W$ satisfies
\[
\tag{Int}\label{cond:Int}
\begin{array}{ll}
\text{if $\VVV \in \operatorname{cols}(\SSS)$:} & \text{$b_{\VVV} + \mathfrak{D}_{\SSS, \sigma}$ occurs with nonzero coefficient in some element of $W$,}\\
\text{if $\VVV \in \operatorname{rows}(\SSS)$:} & \text{$\FF_{q} b_{\VVV} + \mathfrak{D}_{\SSS, \sigma} \not\subseteq W$,}
\end{array}
\]
for each $\VVV \in \operatorname{CR}(\SSS)$.
\end{lem}
\begin{proof}
Take $\mathfrak{n} \le \ut_{n}(\FF_{q})$ so that $\mathfrak{D}_{\SSS, \sigma} \subseteq \mathfrak{n}$ and $\mathfrak{n}/\mathfrak{D}_{\SSS, \sigma} = W$.  As $\mathfrak{n} \subseteq \mathfrak{Z}_{\SSS, \sigma}$, what remains is to determine when $\FF_{q}e_{i, j} \not\subseteq \mathfrak{n}$ for each $\edge{i}{j} \in \nu$ and  $\operatorname{supp}(\mathfrak{n}) = \up{\lambda}$, where $\SSS = \lambda \sqcup \nu$.  These conditions are respectively equivalent to $W$ satisfying~\eqref{cond:Int} for all $\RRR \in \operatorname{rows}(\SSS)$ and all $\CCC \in \operatorname{cols}(\SSS)$.
\end{proof}

Recall that Proposition~\ref{prop:CRsize} gives the concise formula $|\operatorname{CR}(\SSS)| = |\SSS| - 2|\operatorname{bind}(\SSS)|$.

\begin{thm}\label{thm:numsubgroups}
For $q = p^{d}$ with $p$ a prime and $d \in \ZZ_{+}$, 
\[
|\{N \trianglelefteq \UT_{n}(\FF_{q})\}| = \sum_{\SSS \in \mathscr{T}_{n}} (q-1)^{|\operatorname{bind}(\SSS)|} \sum_{i = 0}^{|\operatorname{CR}(\SSS)|} (-1)^{|\operatorname{CR}(\SSS)| - i} \binom{|\operatorname{CR}(\SSS)|}{i}  \sum_{j = 0}^{di} \qbinom{di}{j}_{p},
\]
where $\qbinom{n}{k}_{p}$ denotes the $p$-binomial coefficient.
\end{thm}

\begin{table}
\begin{center}
\begin{tabular}{c||c|c|c|c}
& $n = 2$ &$ n = 3$ & $n = 4$ & $n = 5$
\\ \hline\hline
&&&& \\[-1em]
$d = 1$ & $2$ & $r + 5$ & $3r^{2} + 10r + 14$ & $\makecell{r^{4} + 11r^{3} + 41r^{2} \\ + 62r + 42}$
 \\ \hline
&&&& \\[-1em]
$d = 2$ & $ r + 4 $ & $\makecell{r^{4} + 7r^{3} + 19r^{2} \\ + 25r + 19}$ & $\makecell{r^{9} + 12r^{8} + 64r^{7}+ 204r^{6} \\ + 438r^{5} + 673r^{4} + 756r^{3} \\ + 610r^{2} + 327r + 100}$ & $\makecell{r^{16} + 19r^{15} + 169r^{14} \\ + 938r^{13} + 3653r^{12} \\ + \text{lower order terms}}$ 
%
\end{tabular}
\end{center}
\caption{The value of $|\{N \trianglelefteq \UT_{n}(\FF_{p^{d}})\}|$ as a polynomial in $r = p-1$ for small $n$ and $d$.  The coefficients of each polynomial in this table are positive and unimodal.}
\label{tab:NSGpolynomials}
\end{table}

\begin{proof}[Proof of Theorem~\ref{thm:numsubgroups}]
Fix $(\SSS, \sigma) \in \mathscr{T}_{n}(q)$.  For each $X \subseteq \operatorname{CR}(\SSS)$, there is a bijection
\[
\{W \le \mathfrak{Z}_{\SSS, \sigma}/\mathfrak{D}_{\SSS, \sigma} \;|\; \text{$W$ does not satisfy~\eqref{cond:Int} for $\VVV \notin X$}\} \longleftrightarrow \{W' \le \FF_{q}\operatorname{-span}\{b_{\VVV} \;|\; \VVV \in X\}\}
\]
given by projection onto $\FF_{q}\operatorname{-span}\{b_{\VVV} \;|\; \VVV \in X\}$.  Applying the principal of inclusion-exclusion,
\begin{equation}
\label{eq:numsubgroupsinner}
|\operatorname{NSGfam}(\SSS, \sigma)| =  \sum_{i = 0}^{|\operatorname{CR}(\SSS)|} (-1)^{|\operatorname{CR}(\SSS)| - i} \binom{|\operatorname{CR}(\SSS)|}{i} |\{W' \le \FF_{p}^{di} \}|
\end{equation}
which completes the proof as $\qbinom{di}{j}_{p}$ counts the subgroups $W' \le \FF_{p}^{di}$ with $|W'| = p^{j}$.
\end{proof}

\begin{rems}
\begin{enumerate}[label=(R\arabic*),  ref=(R\arabic*)]
\item When $q = p$ is prime, Theorem~\ref{thm:numsubgroups} and Corollary~\ref{cor:numideals} give the same formula; in particular Equation~\eqref{eq:numsubgroupsinner} is equal to the sum of $q$-Stirling numbers in Corollary~\ref{cor:numideals}.

\item Theorem~\ref{thm:numsubgroups} shows that the number of normal subgroups in $\UT_{n}(\FF_{q})$ is a polynomial in $p$ (considered as an indeterminate); see Table~\ref{tab:NSGpolynomials}.  When $d \neq 1$ so that $q \neq p$, this expression is not always a polynomial in $q$ (e.g. $n = 2$, $d = 2$).

\item Computing $|\{N \trianglelefteq \UT_{n}(\FF_{q})\}|$ for $1 \le n \le 10$ and $1 \le d \le 5$ suggests that this quantity may have positive and unimodal coefficients as a polynomial in $p - 1$.  Corollary~\ref{cor:numideals} implies positivity for $d = 1$ and $n \ge 0$, but unimodality is entirely mysterious.

\end{enumerate}
\end{rems}

\subsubsection{Lattices of normal subgroups}
\label{sec:normalsubgrouplattices}

The ideals of $\ut_{n}(\FF_{q})$ naturally form a lattice, so it is also possible to endow the set
\[
\mathscr{I}_{n}(q) = \{1 + \mathfrak{n} \;|\; \text{$\mathfrak{n}$ is an ideal of $\ut_{n}(\FF_{q})$}\}
\]
with a lattice structure:  $(1 + \mathfrak{n}) \wedge (1+ \mathfrak{m}) = 1 + (\mathfrak{n} \cap \mathfrak{m})$ and $(1 + \mathfrak{n}) \vee (1+ \mathfrak{m}) = 1 + (\mathfrak{n} + \mathfrak{m})$.  The set $\{N \trianglelefteq \UT_{n}(\FF_{q})\}$ is also a lattice, with $\wedge$ and $\vee$ given respectively by intersection and product of subgroups.  A consequence of Theorem~\ref{thm:correspondence} is that these lattice structures agree.

\begin{cor}\label{cor:idealsubgroups}
Let $1 + \mathfrak{m}$ and $1 + \mathfrak{n}$ be normal subgroups of $UT_{n}(\FF_{q})$.  Then
\[
(1 + \mathfrak{n}) (1+ \mathfrak{m}) = 1 + (\mathfrak{n} + \mathfrak{m}).
\]
\end{cor}
\begin{proof}
The product $(1 + \mathfrak{n}) (1+ \mathfrak{m})$ is the intersection of all normal subgroups containing both $1 + \mathfrak{n}$ and $1 + \mathfrak{m}$; by Theorem~\ref{introthm:A} this is the intersection of the sets $1 + \mathfrak{r}$ for additive subgroups $\mathfrak{r}$ of $\ut_{n}$ which contain $\mathfrak{m}$, $\mathfrak{n}$, and $\{[a, b] \;|\; a \in \ut_{n}(\FF_{q}), b \in \mathfrak{r}\}$; this is clearly $1 + (\mathfrak{m} + \mathfrak{n})$.
\end{proof}

An element $M$ of a lattice $\LLL$ is said to be \emph{join irreducible} if there is a unique element of $\LLL$ which is covered by $M$.  For $\LLL = \mathscr{I}_{n}(q)$ or $\{N \trianglelefteq \UT_{n}(\FF_{q})\}$, say that $M \in \LLL$ is \emph{generated by} $1 + a \in \UT_{n}(\FF_{q})$ in $\LLL$ if $M$ is the minimal element of $\LLL$ which contains $1 + a$.  For example, $M = \{1\}$ is generated by $1$ in both lattices.

\begin{prop}\label{prop:singlegenerator}
Let $\LLL = \mathscr{I}_{n}(q)$ or $\{N \trianglelefteq \UT_{n}(\FF_{q})\}$.  A subgroup $M \in \LLL$ is join irreducible if and only if $M \neq \{1\}$ and $M$ is generated in $\LLL$ by an element of $\UT_{n}(\FF_{q})$.
\end{prop}

Proving the proposition requires a description of the subgroups in $\mathscr{I}_{n}(q)$ and $\{N \trianglelefteq \UT_{n}(\FF_{q})\}$ which are generated by any particular element of $\UT_{n}(\FF_{q})$.  For each $1 + a \in \UT_{n}(\FF_{q})$, there is a unique nonnesting set partition $\lambda \in \mathrm{NNSP}_{n}$ for which $\operatorname{supp}(\{a\}) \subseteq \up{\lambda}$ and $a_{i, j} \neq 0$ for each $\edge{i}{j} \in \lambda$.  Let $\KKK$ be the maximal tight splice of $\lambda$ as described in Proposition~\ref{prop:completesplice}, and define a labeling $\theta: \operatorname{bind}(\KKK) \to \FF_{q}^{\times}$ by
\[
\theta\left( \binding{i}{j}{j+1}{l} \right) = a_{j+1, l}/a_{i, j}.
\]

\begin{lem}\label{lem:singlegenerator}
Let $1 + a \in \UT_{n}(\FF_{q})$, and take $\lambda, \KKK$, and $\theta$ as defined above for $1 + a$.
\begin{enumerate}[label=(\roman*)]
\item The elements of $\UT_{n}(\FF_{q})$ which generate the same normal subgroup as $1 + a$ does are exactly those in the set $1 + \FF_{p}^{\times}a + \mathfrak{D}_{\KKK, \theta}$, and this subgroup is  $1 + \FF_{p}a + \mathfrak{D}_{\KKK, \theta}$.

\item The elements of $\UT_{n}(\FF_{q})$ which generate the same element of $\mathscr{I}_{n}(q)$ as $1 + a$ does are exactly those in the set $1 + \FF_{q}^{\times}a + \mathfrak{D}_{\KKK, \theta}$, and this subgroup is $1 + \FF_{q}a + \mathfrak{D}_{\KKK, \theta}$.

\end{enumerate}
\end{lem}
\begin{proof}
Claims (i) and (ii) follow from nearly identical arguments, so I will only present the proof of (i).  From the definition of $\KKK$ and $\theta$, it is the case that $a \in \mathfrak{Z}_{\KKK, \theta}$.  Lemma~\ref{lem:intermediatesubgroup} therefore implies that $1 + \FF_{p}a + \mathfrak{D}_{\KKK, \theta} \trianglelefteq \UT_{n}(\FF_{q})$, so what remains is to show that $1 + \FF_{p}a + \mathfrak{D}_{\KKK, \theta}$ is the minimal normal subgroup containing $1 + a$.

Suppose that $N$ is the minimal normal subgroup of $\UT_{n}(\FF_{q})$ which contains $a$, so that $N \subseteq 1 + \FF_{p}a + \mathfrak{D}_{\KKK, \theta}$.  It follows from the definition of $N$ that $\operatorname{supp}(N) = \up{\lambda}$, so by Corollary~\ref{cor:grpsplicecondition}, $N$ properly contains $1 + \mathfrak{D}_{\SSS, \sigma}$ for some $(\SSS, \sigma) \in \mathscr{T}_{\lambda}(q)$. Thus
\[
|N| > |\mathfrak{D}_{\SSS, \sigma}| = q^{|\upnum{\lambda}{1}| - |\operatorname{rows}(\SSS)|} \ge q^{|\upnum{\lambda}{1}| - |\operatorname{rows}(\KKK)|} = |\mathfrak{D}_{\KKK, \theta}|
\]
since $\operatorname{rows}(\SSS) \subseteq \operatorname{rows}(\KKK)$.  The order of $N$ is a power of $p$, so $|N| \ge p|\mathfrak{D}_{\KKK, \theta}| = |1 + \FF_{p}a + \mathfrak{D}_{\KKK, \theta}|$, and thus $N = 1 + \FF_{p}a + \mathfrak{D}_{\KKK, \theta}$.  Generalizing this argument, every element of $1 + \FF_{p}^{\times}a + \mathfrak{D}_{\KKK, \theta}$ generates $1 + \FF_{p}a + \mathfrak{D}_{\KKK, \theta}$.

Finally, suppose that $1 + b \in \UT_{n}(\FF_{q})$ also generates $1 + \FF_{p}a + \mathfrak{D}_{\KKK, \theta}$ as a normal subgroup.  As $1 + \mathfrak{D}_{\KKK, \theta}$ is normal, this must mean that $1 + b \in 1 + \FF_{p}^{\times}a + \mathfrak{D}_{\KKK, \theta}$.
\end{proof}

\begin{proof}[Proof of Proposition~\ref{prop:singlegenerator}]
If $M \in \LLL$ is join irreducible, then $M$ is generated in $\LLL$ by each element of $M \setminus K$, where $K \in \LLL$ is the unique element covered by $M$.  If $M$ is generated by an element $1 + a \in \UT_{n}(\FF_{q})$, then $M$ is one of the subgroups in Lemma~\ref{lem:singlegenerator}.  For $a \neq 0$, these subgroups cover only $\mathfrak{D}_{\KKK, \theta}$ in their respective lattices, and are therefore join irreducible.
\end{proof}

\printbibliography

\end{document}